\documentclass{amsart}
\usepackage{amssymb}
\usepackage{appendix}
\usepackage{euscript,eufrak,verbatim}
\usepackage[symbol]{footmisc}
\usepackage{a4wide}
\usepackage{graphicx}
 \usepackage[all]{xy}
 \usepackage{enumerate}
\usepackage{xcolor}
\usepackage{stmaryrd}
\usepackage{boxhandler}
\usepackage{caption}
\usepackage{lipsum}
\newcommand{\vertiii}[1]{{\left\vert\kern-0.25ex\left\vert\kern-0.25ex\left\vert #1 
    \right\vert\kern-0.25ex\right\vert\kern-0.25ex\right\vert}}

\newtheorem*{theorem*}{Main Theorem}
\newtheorem*{theorem**}{Theorem}
\newtheorem{theorem}{Theorem}

\newtheorem{lem}{Lemma}
\newtheorem*{conj}{Conjecture}

\newtheorem{prop}[theorem]{Proposition}
\newtheorem{coro}{Corollary}
\newtheorem{rem}[theorem]{Remark}
\theoremstyle{definition}

\newtheorem*{ques}{Question}

\newtheorem*{Rep}{Reparametrization Lemma}

\newcommand{\Jac}{\mathop{\mathrm{Jac}}}

\newcommand{\diam}{\mathop{\mathrm{diam}}}

\newcommand{\Leb}{\mathop{\mathrm{Leb}}}
\newcommand{\N}{\mathbb{N}}
\theoremstyle{remark}

\newcommand{\G}{\llbracket}
\newcommand{\R}{\rrbracket}
\numberwithin{equation}{section}

\begin{document}
\begin{large}
\title{SRB measures for $C^\infty$ surface diffeomorphisms}

\author{David Burguet}
\address{CNRS, Sorbonne Universite, LPSM, 75005 Paris, France}
  \email{david.burguet@upmc.fr}  

\date{October 2021}


\begin{abstract}A $C^\infty$ surface diffeomorphism admits a SRB measure if and only if the set $\{ x, \ \limsup_n\frac{1}{n}\log \|d_xf^n\|>0\}$  has positive Lebesgue measure. Moreover the basins of the ergodic SRB measures are covering this set Lebesgue almost everywhere. We also obtain similar results for $C^r$ surface diffeomorphisms with  $+\infty>r>1$.
\end{abstract}

\maketitle
\tableofcontents

\section{Introduction}One fundamental problem in dynamics consists in 
understanding the statistical behaviour of the system. Given a topological system $(X,f)$ we are more precisely interesting in  the asymptotic distribution of the empirical measures $\left(\frac{1}{n}\sum_{k=0}^{n-1}\delta_{f^kx}\right)_n$ for typical points $x$ with respect to a reference measure.  In the setting of differentiable dynamical systems the natural reference measure to consider is the Lebesgue measure on the manifold. 

The basin of a $f$-invariant measure $\mu$ is the set $\mathcal B(\mu)$ of points whose empirical measures are converging to $\mu$ in the weak-$*$ topology. By Birkhoff's ergodic theorem the basin of an  ergodic  measure $\mu$ has full $\mu$-measure.  An invariant measure is said physical when its basin has positive Lebesgue measure. We may  wonder when such measures exist and then study their basins. 

In the works of Y. Sinai, D. Ruelle and R. Bowen \cite{sin,bow,rue} these questions have been successfully solved for uniformly hyperbolic systems.
A SRB measure of a $C^{1+}$ system  is an invariant probability measure with at least one positive Lyapunov exponent almost everywhere, which has absolutely continuous conditional measures on unstable manifolds  \cite{young}.  Physical measures may neither be SRB measures nor sinks (as in the famous figure-eight attractor), however hyperbolic ergodic SRB measures are physical measures. For uniformly hyperbolic systems, there is a finite number of such measures and their basins cover a full Lebesgue subset of the manifold.
 Beyond the uniformly hyperbolic case such a picture is also known for large classes of partially hyperbolic systems  \cite{BV,ABV,ADLP}.   Corresponding results have been established for unimodal maps with negative Schwartzian derivative \cite{kel}. SRB measures have  been also deeply  investigated for parameter families such as the quadratic family  and Henon maps \cite{jak, BC,BY,VB}.
 In his celebrated ICM's talk, M. Viana conjectured that a surface diffeomorphism admits a SRB measure, whenever the set of points with positive Lyapunov exponent has positive Lebesgue measure.
 In recent works some weaker versions of the conjecture (with some additional assumptions of recurrence  and Lyapunov regularity) have been proved \cite{cli,BO}. Finally we mention that in the present context of $C^\infty$ surface diffeomorphims  J. Buzzi, S. Crovisier, O. Sarig have also recently shown   the existence of  a SRB measure when the set of points with a positive Lyapunov exponent has positive Lebesgue measure \cite{BCS3} (Corollary \ref{coco}).  \\
 
In this paper we define a general entropic approach  to build SRB measures, which we apply to prove Viana's conjecture for $C^\infty$ surface diffeomorphisms. We strongly believe the same approach may be used to recover the existence of SRB measures  for weakly mostly expanding partially hyperbolic systems \cite{ADLP} and to give another proof of Ben Ovadia's criterion for $C^{1+}$ diffeomorphisms in any dimension  \cite{BO}. \\


We state now the main results of our paper.   Let $(M,\|\cdot \|)$ be a compact Riemannian surface and let $\Leb$ be a volume form on  $M$, called  Lebesgue measure. 
We consider  a $C^\infty$ surface diffeomorphism  $f:M\circlearrowleft$. 
The maximal Lyapunov exponent at $x\in M$ is given by $\chi(x)=\limsup_n\frac{1}{n}\log \|d_xf^n\|.$
When  $\mu$  is a $f$-invariant  probability measure, we let 
$\chi(\mu)=\int \chi(x)\, d\mu(x).$
For   two Borel subsets $A$ and $B$ of $M$ we write  $A\stackrel{o}{\subset} B$ (resp. $A\stackrel{o}{=} B$) when we have $\Leb(A\setminus B)=0$ (resp. $\Leb(A\Delta B)=0$).

\begin{theorem} 
Let $f:M\circlearrowleft$ be a $C^\infty$ surface diffeomorphism. There are countably many ergodic SRB measures $(\mu_i)_{i\in I}$, such that we have with $\Lambda=\{\chi(\mu_i), \ i\in I\}\subset \mathbb R_{>0}$:
\begin{itemize}
\item   $\{\chi>0\}\stackrel{o}{=} \{\chi\in \Lambda\}$, 
\item  $\{\chi=\lambda\}\stackrel{o}{\subset}\bigcup_{i, \chi(\mu_i)=\lambda}\mathcal {B}(\mu_i) $ for all $\lambda\in \Lambda$.
\end{itemize} 

\end{theorem}

\begin{coro}
Let $f:M\circlearrowleft$ be a $C^\infty$ surface diffeomorphism. Then 
$$\{\chi>0\}\stackrel{o}{\subset}\bigcup_{\mu \text{ SRB ergodic}}\mathcal {B}(\mu).$$
\end{coro}

\begin{coro}[Buzzi-Crovisier-Sarig \ \cite{BCS3}]\label{coco}
Let $f:M\circlearrowleft$ be a $C^\infty$ surface diffeomorphism. 

If  $\Leb(\chi>0)>0$, then there exists a SRB measure.

\end{coro}

 In fact we  establish a $C^r$, $1<r< +\infty$, stronger version, which implies straightforwardly   Theorem 1:
 
 \begin{theorem*} \label{Cr}
Let $f:M\circlearrowleft$ be a $C^r$, $r>1$,  surface diffeomorphism. Let $R(f):=\lim_n\frac{1}{n}\log^+ \sup_{x\in M}\|d_xf^n\|$.     There are countably many ergodic SRB measures $(\mu_i)_{i\in I}$ with $\Lambda:=\{\chi(\mu_i), \ i\in I\}\subset ]\frac{R(f)}{r},+\infty[$, such that we have :
\begin{itemize}
\item   $\left\{\chi>\frac{R(f)}{r}\right\}\stackrel{o}{=} \{\chi\in \Lambda\}$, 
\item  $\{\chi=\lambda\}\stackrel{o}{\subset}\bigcup_{i, \chi(\mu_i)=\lambda}\mathcal {B}(\mu_i) $ for all $\lambda\in \Lambda$.
\end{itemize} 

\end{theorem*}
 
              When $f$ is a $C^{1+}$ topologically transitive surface diffeomorphism,                there is at most one SRB measure, i.e. $\sharp I\leq 1$ \cite{RRHH}.  If moreover  the system is topologically mixing, then the SRB measure when it exists is Bernoulli \cite{BCS}. By  the spectral decomposition of  $C^r$ surface diffeomorphisms for $1<r\leq +\infty$ \cite{BCS} there are at most finitely many ergodic SRB measures with entropy and thus maximal  exponent larger than  a given  constant $b>\frac{R(f)}{r}$. Therefore, in the Main Theorem,  the set  $\Lambda=\{\chi(\mu_i), \ i\in I\}$ is either finite or a sequence decreasing to $\frac{R(f)}{r}$.  When $r$ is finite, there may also exist ergodic SRB measures $\mu$ with $\chi(\mu)\leq \frac{R(f)}{r}$. \\

            We prove in a forthcoming paper \cite{burex} that  the above statement is sharp by  building for any finite $r>1$ a $C^r$ surface diffeomorphism $(f,M)$ with a periodic saddle hyperbolic point  $p$ such that $\chi(x)=\frac{R(f)}{r}>0$ for all $x\in U$ for some set $U\subset \mathcal B(\mu_p)$ with $\Leb(U)>0$, where $\mu_p$ denotes the periodic measure associated to $p$  (see \cite{bur} for such an example of interval maps). \\

In higher dimensions we let $\Sigma^k\chi(x):=\limsup_n\frac{1}{n}\|\Lambda^k d_xf^n\|$ where $\Lambda^k df$ denotes the action induced by $f$ on the $k^{th}$ exterior power of $TM$ for $k=1,\cdots, d$ with $d$ being the dimension of $M$.   By convention we also let $\Sigma^0\chi=0$. For any $C^1$ diffeomorphism $(M,f)$ we have $\Leb(\Sigma^d\chi>0)=0$ (see \cite{ara}). 
The product of a figure-eight attractor with a surface Anosov diffeomorphism does not admit any SRB measure whereas $\chi$ is positive on a set of positive Lebesgue measure. However we conjecture :

\begin{conj}
Let $f:M\circlearrowleft$ be a $C^\infty$ diffeomorphism on a compact manifold (of any dimension). 

If $\Leb\left(\Sigma^k\chi>\Sigma^{k-1}\chi\geq 0\right)>0$, then there exists an ergodic  measure with at least $k$ positive Lyapunov exponents, such that its entropy is larger than or equal to the sum of its $k$ smallest positive Lyapunov exponents. 
\end{conj}

In the present two-dimensional case the semi-algebraic tools used to bound the distorsion and the local volume growth of $C^\infty$ curves are elementary. This is  a challenging problem to adapt this   technology  in higher dimensions. \\

When the empirical measures from   $x\in M$ are not converging,  the point  $x$ is said to have  historic behaviour \cite{ruee}.   A set $U$ is    contracting when  the diameter of $f^nU$ goes to zero when $n\in \N$ goes to infinity.  In a contracting set the empirical measures of all points have the same limit set,  however  they may not converge. P. Berger and S. Biebler have  shown  that   $C^\infty$ densely inside the  Newhouse domains \cite{BB} there are contracting  domains with historic behaviour. In intermediate smoothness, such domains have been previously built in \cite{KS}.  As a consequence of the Main Theorem, Lebesgue almost every point $x$ with historic behaviour  satisfies $\chi(x)\leq 0$ for $C^\infty$ surface diffeomorphisms. We also show the following statement.

\begin{theorem}\label{yoyo}
Let $f$ be a $C^\infty$ diffeomorphism on a compact manifold (of any dimension).  Then 
 Lebesgue a.e. point $x$ in a contracting set satisfies $\chi(x)\leq 0$. 
\end{theorem}

\begin{ques}
Let $f$ be a $C^\infty$ surface diffeomorphism. Assume the set $H$ of points with historic behaviour has positive  Lebesgue measure. Does every Lebesgue density point of $H$ belongs  to a almost contracting\footnote{See Section \ref{gui} for the definition of almost contracting set.} set with positive Lebesgue measure? \\
\end{ques}


We explain now  in few lines the main ideas to build a SRB measure under the assumptions of the Main Theorem. The geometric approach for uniformly hyperbolic systems consists in considering a weak limit of $\left(\frac{1}{n}\sum_{k=0}^{n-1}f_*^k\Leb_{D_u}\right)_n$, where $D_u$ is a local unstable disc and  $\Leb_{D_u}$ denotes the normalized Lebesgue measure on $D_u$ induced by its inherited Riemannian structure as a submanifold of $M$.  Here we take a smooth $C^r$ embedded curve $D$ such that 
$$\chi(x,v_x):=\limsup_n\frac{1}{n}\log \|d_xf^n(v_x)\|>b>\frac{R(f)}{r}$$ for  $(x,v_x)$ in the unit tangent space $T^1D$ of $D$ with $x$ in a subset $B$ of $D$ with positive $\Leb_D$-measure.  For $x$ in $B$ we define  a subset $E(x)$ of  positive integers, called the \textit{ geometric set},  such that  the following properties hold for any $n\in  E(x)$  :
\begin{itemize}
\item the geometry of $f^nD$ around $f^nx$ is \textit{bounded} meaning that for some uniform  $\epsilon>0$, the connected component $D_n^\epsilon(x)$ of $f^nD$ with the ball at $f^nx$ of radius $\epsilon>0$ is a curve with bounded $s$-derivative for $s\leq r$,
\item the distorsion of $df^{-n}$ on the tangent space of $ D_n^\epsilon(x)$ is controlled, 
\item for some $\tau>0$ we have $\frac{\|d_xf^{l}(v_x)\|}{\|d_xf^{k}(v_x)\|}\geq  e^{(l-k)\tau}$ for any $l>k\in E(x)$.
\end{itemize}
We show that $E(x)$ has positive upper asymptotic density for $x$ in a subset $A$ of $B$ with positive $\Leb_D$-measure.  Let $F:\mathbb PTM\circlearrowleft$ be the map induced by $f$ on the projective tangent bundle $\mathbb PTM$.
We build a SRB measure by considering a weak limit $\mu$ of a sequence of the form  $\left(\frac{1}{\sharp F_n}\sum_{k\in F_n}F_*^k\mu_n\right)_n$ such that :
\begin{itemize}
\item $(F_n)_n$ is a F\"olner sequence,  so that the weak limit $\mu$ will be invariant by $F$,
\item for all $n$, the measure $\mu_n$ is the probability measure induced by $\Leb_D$ on $A_n\subset A$,  the $\Leb_D$-measure of $A_n$ being  not exponentially small,
\item the sets $(F_n)_n$ are in some sense  \textit{filled with } the geometric set  $E(x)$  for $x\in A_n$. Then the measure $\mu$ on $\mathbb PTM$ will be supported on the unstable Oseledec's bundle.
\end{itemize} 
Finally we check with some F\"olner Gibbs property that the limit empirical measure $\mu$ projects to  a SRB measure on $M$ by using the Ledrappier-Young entropic characterization.\\

The paper is organized as follows. In Section 2 we recall for general sequences of integers  the notion of asymptotical density and we build for any  sequence $E$ with positive upper density a F\"olner  set $F$ filled with  $E$. Then we use a Borel-Cantelli argument  to define our  sets $(A_n)_n$ and the F\"olner sequence $(F_n)_n$. In Section 3, we study the maximal  Lyapunov exponent and the entropy of  the generalized empirical measure $\mu$ assuming some Gibbs property. We introduce the geometric set  in Section 4 by using the Reparametrization Lemma of \cite{bure}. We build then SRB measures  in Section 5 by using the abstract formalism of Section 2 and 3.  Then we prove the covering property of the basins in Section 6 by  the standard argument of absolute continuity of Pesin stable foliation.   The last section is devoted to the proof of  Theorem \ref{yoyo}. \\

\textit{Comment :} In a first version of this work, by following \cite {bure}  (incorrectly) the author claimed that, at $b$-hyperbolic  times $n$ of the sequence $\left(\|d_xf^k(v_x)\|\right)_k$ for some $b>0$, the geometry of $f^nD$ at $f^nx$ was bounded.  J. Buzzi, S. Crovisier and O. Sarig gave then in  \cite{BCS3} another proof of Corollary \ref{coco} by using their analysis of the entropic continuity of Lyapunov exponents from \cite{BCS2}.
But as realized recently, our claim on the geometry at hyperbolic times is wrong in general and  we manage to show it  only when $\chi(x)>\frac{R(f)}{2}$.  In this last version, we correct our proof by showing directly that the set of times with bounded geometry has positive upper asymptotic density on a set of positive $\Leb_D$-measure. Our proof  is still based on the Reparametrization Lemma proved in \cite{bure}.

\section{Some asymptotic properties of integers}\label{drei}

\subsection{Asymptotic density}
We first introduce some notations.
In the following we let $\mathcal P_\N$ and $\mathcal P_n$ be respectively the   power sets of $\mathbb N$ and $\{1,2,\cdots,n\}$,  $n\in \N$.  
The\textbf{ boundary} $\partial E$ of $E\in \mathcal P_\N$ is the subset of $\N$ consisting in the integers $n\in E$ with $n-1\notin E$ or  $n+1\notin E$. We also let $E^{-}:=\{n\in E, \  n+1\in E\}$. For $a,b\in \N$ we write $\G a,b \R$  (resp. $\G a,b \G$, $\R a,b\R$) the interval of 
integers $k$ with $a\leq k\leq b$ (resp $a\leq k< b$, $a<k\leq b$). The \textbf{connected components} of $E$ are the maximal  intervals of integers contained in $E$. An interval of integers $\G a, b\G$ is said $E$-\textbf{irreducible} when we have $a, b\in E $ and $\G a,b \G\cap E=\{a\}$.
 For $E\in \mathcal P_\N$ we let $E_{(n)}:=E\cap \G 1,n \R\in \mathcal P_n$ for all $n\in \N$.  For  $M\in \N$,  we   denote by $E_M$     the union of the  intervals $\G a,b \R$    with $a,b\in E$ and $|a-b|\leq M$.

We let $\mathfrak N$  be the set of increasing sequences of natural  integers,  which may be identified with the subset of $\mathcal P(\N)$ given by infinite subsets of $\N$.  For $\mathfrak n\in \mathfrak N$ we define the \textbf{generalized power set of $\mathfrak n$} as 
$\mathcal Q_\mathfrak n:=\prod_{n\in \mathfrak n} \mathcal P_n.$\\

We recall now the classical notion of upper and lower asymptotic densities. 
For $n\in \N^*$ and $F_n\in \mathcal P_n$ we let  $d_n(F_n)$ be the frequency of $F_n$ in $\G 1,n \R$: $$d_n(F_n)=\frac{\sharp F_n}{n}.$$
The \textbf{upper and lower asymptotic densities} $\overline{d}(E)$  and $\underline{d}(E)$ of  $E\in \mathcal P_\N$ are respectively defined by
$$\overline{d}(E):=\limsup_{n\in \mathbb N}d_n(E_{(n)} ) \text{ and} $$
$$\underline{d}(E):=\liminf_{n\in \mathbb N}d_n(E_{(n)} ). $$
We just write $d(E)$ for the limit,  when the frequencies $d_n(E_{(n)})$ are converging.
For any $\mathfrak n\in \mathfrak N$ we let similarly   $\overline{d}^\mathfrak n(E):=\limsup_{n\in \mathfrak n}d_n(E_{(n)} )$ and $\underline{d}^\mathfrak n(E):=\liminf_{n\in \mathfrak n}d_n(E_{(n)} )$.  The concept of  upper and lower  asymptotic densities of $E\in \mathcal P_\N$ may be extended to generalized power sets  as follows.  For $\mathfrak n\in \mathfrak N$ and   $\mathcal F=(F_n)_{n\in \mathfrak n}\in\mathcal Q_\mathfrak n$ we let 
$$\overline{d}^{\mathfrak n}(\mathcal F):=\limsup_{n\in \mathfrak n}d_n(F_n) \text{ and }$$
$$\underline{d}^{\mathfrak n}(\mathcal F):=\liminf_{n\in \mathfrak n}d_n(F_n). $$
Again we just write $d^{\mathfrak n}(E) $ and $d^{\mathfrak n}(\mathcal F)$ when the corresponding frequencies are converging.

\subsection{F\"olner sequence and density along subsequences}
We say that $E\in \mathcal P_\N$ is \textbf{F\"olner  along} $\mathfrak n\in \mathfrak N$ when its boundary $\partial E$ has zero upper asymptotic density with respect to   $\mathfrak n$, i.e. $\overline{d}^\mathfrak n(\partial E)=0$.  More generally $\mathcal F=(F_n)_{n\in \mathfrak n}\in \mathcal Q_\mathfrak n $ with $\mathfrak n\in \mathfrak N$ is F\"olner when  we have $\overline{d}^\mathfrak n(\partial \mathcal F)=0$ with $\partial \mathcal F=(\partial F_n )_{n\in \mathfrak n}$. In general this property seems to be weaker than the usual F\"olner property $\limsup_{n\in \mathfrak n}\frac{\sharp \partial F_n}{\sharp F_n}=0$. But in the following we will work with sequences $\mathcal F$ with $\underline{d}^{\mathfrak n}(\mathcal F)>0$. In this case our definition coincides with the standard one. \\

Let $E,F\in \mathcal P_\N$ and $\mathfrak n\in \mathfrak N$.  We say  that $F$ is \textbf{ $\mathfrak n$-filled with $E$} or $E$ is \textbf{dense in $F$ along $\mathfrak n$}  when we have 
$$\overline{d}^{\mathfrak n}(F\setminus E_M)\xrightarrow{M\rightarrow +\infty}0.$$
Observe that $\left(\overline{d}(E_M)\right)_M$ is converging nondecreasingly to some $a\geq \overline{d}(E)$ when $M$ goes to infinity.  The limit $a$ is in general strictly less than $1$. For example if $E:=\bigcup_n \G 2^{2n},2^{2n+1}\R$ one easily computes $\overline{d}(E_M)=\overline{d}(E)=2/3$ for all $M$. In this case, the set $E$ is moreover a  F\"olner set.\\

Also $\mathcal F=(F_n)_{n\in \mathfrak n}\in \mathcal Q_\mathfrak n
$ is said filled with $E$ when   we have with $\mathcal F\setminus E_M:=(F_n\setminus E_M)_{n\in \mathfrak n}$:
$$\overline{d}^{\mathfrak n}(\mathcal F\setminus E_M)\xrightarrow{M\rightarrow +\infty}0.$$ 

\subsection{F\"olner set $F$ filled with a given $E$ with $\overline{d}(E)>0$}
Given a set $E$ with positive upper asymptotic density we build a F\"olner set $F$ filled with $E$ by using a diagonal argument. More precisely we will build $F$ by filling the holes in $E$ of larger and larger size when going to infinity.

\begin{lem}
For  any $E$ with $\overline{d}(E)>0$ there is a subsequence $\mathfrak n\in \mathfrak N$  and $F\in \mathcal P_\N$ with $\partial F\subset E$ such that 
\begin{itemize}
\item $d^\mathfrak{n}(F)\geq d^{\mathfrak n}(E\cap F )= \overline{d}(E)$;
\item $F $ is F\"olner along $\mathfrak n$;
\item $E$ is dense in $F$ along $\mathfrak n$.  
\end{itemize}
\end{lem}

\begin{proof} We first consider a subsequence  $\mathfrak n^0=(\mathfrak n^0_k)_k$ satisfying $d^{\mathfrak n^0}(E)=\overline{d} (E)$.  We can ensure  
that $\mathfrak n^0_k$ belong to $E$ for all $k$. Observe that $\overline{d}(E\setminus E_M)\leq 1/M$ for all $M\in \N^*$. Therefore for $M>2/\overline{d}(E)$,  we have $\overline{d}^{\mathfrak n^0}(E_{M})\geq \overline{d}^{\mathfrak n^0}(E_{M}\cap E)>\overline{d}(E)/2>0$.
 We fix such an integer $M$ and  we extract again a subsequence 
$\mathfrak{n}^{M}=(\mathfrak{n}^{M}_k)_k$ of $\mathfrak n^0$   such that 
$d^{\mathfrak{n}^{M}}(E_{M})$ is a limit equal to $\Delta_{M}:=\overline{d}^{\mathfrak n^0}(E_{M})$.  Then we put 
$\Delta_{M+1}=\overline{d}^{\mathfrak{n}^{M}}(E_{M+1})$ and we consider a subsequence $
\mathfrak n^{M+1}$ of $\mathfrak n^{M}$ such that $d^{\mathfrak{n}^{M+1}}(E_{M+1})
$ is  a limit equal to $\Delta_{M+1}\geq \Delta_{M}$ and $d_l(E_{M+1})\leq 
\Delta_{M+1}+1/2^{M+1}$ for all $l\in \mathfrak{n}^{M+1}$. We define by induction in this way nested  sequences $
\mathfrak n^k$ for $k> M$ such that $d^{\mathfrak{n}^{k}}(E_{k})$ is a limit equal to 
$\Delta_{k}=\overline{d}^{\mathfrak{n}^{k-1}}(E_{k})$ and $d_l(E_{k})\leq \Delta_{k}+1/2^{k}$ for all $l\in \mathfrak n^{k}$. We let $\Delta_\infty>\overline{d}(E)/2>0$ be the limit of the nondecreasing sequence $(\Delta_k)_k$. We consider  the diagonal sequence $\mathfrak n=(\mathfrak n_k)_{k\geq M}= (\mathfrak n^k_k)_{k\geq M}$ and we let
$$F=\bigcup_{k>M}\G \mathfrak n_{k-1},\mathfrak n_k\R \cap E_k.$$  Clearly we have $\partial F\subset \mathfrak n^0\cup E\subset E$. 

 On the one hand, $F\cap \G 1,\mathfrak n_k\R$ is contained in $E_k\cap \G  1,\mathfrak n_k\R$ so that
  \begin{align*}d_{\mathfrak n_k}(F)&\leq   d_{\mathfrak n^k_k}(E_k),\\& \leq \Delta_k+1/2^k,\\ \overline{d}^{\mathfrak n}(F) & \leq \lim_k\Delta_k=\Delta_{\infty},
  \end{align*}
  On the other hand,  $F\cap \G 1,\mathfrak n_k\R  $ contains $E_l\cap \G \mathfrak n_{l-1},\mathfrak n_k\R$ for all $M<l<k$. Therefore \begin{align*}d_{\mathfrak n^k_k}(E_l)-\frac{\mathfrak n_{l-1}}{\mathfrak n^k_k}&\leq d_{\mathfrak n_k}(F),\\
   \Delta_l&\leq  \underline{d}^{\mathfrak n}(F),\\
   \Delta_{\infty}&\leq \underline{d}^{\mathfrak n}(F).
  \end{align*}
   We conclude  $d^{\mathfrak n}(F)=\Delta_\infty$. 
 
Similarly we have  for all $l>M$ :
 \begin{align*}
 \underline{d}^{\mathfrak n}(E\cap F)&\geq  \underline{d}^{\mathfrak n}(E\cap E_l),\\
 &\geq d^{\mathfrak n}(E)-1/l,\\
   &\geq \overline{d}^{\mathfrak{m}}(E)-1/l,
 \end{align*}
 therefore $ \underline{d}^{\mathfrak n}(E\cap F)\geq \overline{d}(E)$. Also $\overline{d}^{\mathfrak n}(E\cap F)\leq  d^{\mathfrak n}(E)= \overline{d}^\mathfrak{m}(E)$. Consequently we get $d^{\mathfrak n}(E\cap F)=\overline{d}^\mathfrak{m}(E)$.

We check now that $E$ is dense in $F$. For $l$ fixed and for all $k\geq l$ we have
\begin{align*}
d_{\mathfrak n_k}(F\setminus E_l)&\leq d_{\mathfrak n^k_k}(E_k\setminus E_l),\\
& \leq d_{\mathfrak n_k^k}(E_k)-d_{\mathfrak n_k^k}(E_l),\\
&\leq \Delta_k+1/2^k -d_{\mathfrak n_k^k}(E_l).
\end{align*}
By taking the limit in $k$, we get $\overline{d}^{\mathfrak n}(F\setminus E_l)\leq \Delta_\infty-\Delta_l\xrightarrow{l}0$.

Let us prove finally the F\"olner property of the set $F$.  For  $\mathfrak n_k<K\in \partial F$ either $[K-k,K[$ or $]K,K+k]$ lies in the complement of $E$. Therefore $\overline{d}^{\mathfrak n}(\partial F)\leq 2/k$. As it holds  for all $k$, the set  $F$ is F\"olner along $\mathfrak n$. 
\end{proof}

\subsection{Borel-Cantelli argument}
Let $(X, \mathcal A, \lambda)$ be a measure space with $\lambda$ being a finite measure.  A map $E:X\rightarrow \mathcal P_\N$ is said measurable, 
 when for all $n\in \N$ the set  $\{x, \ n\in E(x)\}$ belongs to $\mathcal A$ (equivalently writing $E$ as an increasing sequence $(n_i)_{i\in \N}$ the integers valued functions $n_i$ are measurable).  For such measurable  maps $E$ and $\mathfrak n$,  the upper asymptotic density $\overline{d}^\mathfrak n(E)$ defines a measurable function.

\begin{lem}\label{measurable}Assume $E$ is a  measurable sequence of integers such that $\overline{d}(E(x))>\beta>0$ for $x$ in a measurable set $A$ of a positive $\lambda$-measure. Then there exist  $\mathfrak n\in \mathfrak N$, measurable subsets $(A_n)_{n\in \mathfrak n}$ of $X$ 	and $\mathcal F=(F_n)_{n\in \mathfrak n}\in \mathcal Q_{\mathfrak n}$ with $\partial F_n\subset E(x)$ for all $x\in A_n$, $n\in \mathfrak n$    such that :
\begin{itemize}
\item $\underline{d}^\mathfrak n(\mathcal F)\geq \beta$; 
\item   $\lambda(A_n)\geq \frac{e^{-n\delta_n}}{n^2}$  for all $n\in \mathfrak n$ with $\delta_n\xrightarrow{\mathfrak n \ni n\rightarrow +\infty}0$;
\item $\mathcal F$ is a F\"olner sequence;
\item $E$ is dense in $\mathcal F$ uniformly on $A_n$, i.e.
$$\limsup_{n\in \mathfrak n}\sup_{x\in A_n}d_n\left( F_n\setminus E_M(x)\right) \xrightarrow{M}0.$$
\item \begin{align*}
\liminf_{n\in \mathfrak n} \inf_{x\in A_n}d_n\left(E(x)\cap F_n\right)& \geq \beta.
\end{align*}
\end{itemize}
\end{lem}

\begin{proof}
The sequences  $\mathfrak n$ and $F$  built in the previous lemma define measurable sequences on $A$. By taking a smaller subset $A$ we may assume 
\begin{itemize}
\item $\mathfrak n_k(x)$ is bounded on $A$ for all $k$,
\item  $d_{\mathfrak n_k(x)}(\partial F(x))\xrightarrow{k}0$  uniformly in $x\in A$,
\item $\limsup_k \sup_{x\in A}d_{\mathfrak n_k(x)}(F(x)\setminus E_M(x))\xrightarrow{M}0$,
\item $d_{\mathfrak n_k(x)}(E(x)\cap F(x))\xrightarrow{k}d^{\mathfrak n(x)}(E(x)\cap F(x))\geq \beta$  uniformly in $x\in A$.
\end{itemize}   By Borel-Cantelli Lemma, the subset $A_n:=\{x\in A, \ n\in \mathfrak n(x)\}$ has $\lambda$-measure larger than $1/n^2$ for infinitely many $n\in \mathbb N$. We let $\mathfrak n$ be this infinite subset of integers.  By the (uniform in $x$) F\"olner property of $F(x)$,  the cardinality  of the boundary of  $\left(F(x)\right)_{(n)}=F(x)\cap \G 1, n\R$ for $x\in A_n$ and $n\in \mathfrak n$ is less than $n\alpha_n$ for some sequence  $(\alpha_n)_{n\in \mathfrak n}$  (independent of $x$) going to $0$.
Therefore there are at most $2\sum_{k=1}^{[n\alpha_n]}{n \choose k}$ choices for $(F(x))_{(n)}$ and thus it may be fixed by  dividing the measure of $A_n$ by $2\sum_{k=1}^{[n\alpha_n]}{n \choose k}=e^{n\delta_n}$  for some $\delta_n\xrightarrow{n}0$. 
\end{proof}


\section{Empirical measures associated to F\"olner sequences}\label{zwei}
Let $(X,T)$ be a topological system, i.e. $X$ is a compact metrizable space and  $T:X\circlearrowleft$ is continuous.  We denote by $\mathcal M(X)$ the set of Borel probability measures on $X$ endowed with the weak-$*$ topology and by $\mathcal M(X,T)$ the compact subset of invariant measures. We will write  $\delta_x$  for the Dirac measure at $x\in X$.   We let $T_*$ be the induced  (continuous) action on $\mathcal M(X)$.   For $\mu\in \mathcal M (X)$ and a finite subset $F$ of $\N$,   we let $\mu^F$ be the empirical measure $\mu^F:=\frac{1}{\sharp F }\sum_{k\in F}T_*^k\mu$. 

\subsection{Invariant measures}
The following lemma is standard, but we give a proof for the sake of completeness. We fix  $\mathfrak n\in \mathfrak N$ and $\mathcal F=(F_n)_{n\in \mathfrak N}\in \mathcal Q_{\mathfrak n}$. 
\begin{lem}\label{fol}
Assume $\mathcal F$ is a  F\"olner sequence   and $\underline{d}^{\mathfrak n}(\mathcal F)>0$.  Let $(\mu_n)_{n\in \mathfrak n}$ be a family in $\mathcal M(X)$ indexed by $\mathfrak n$.  Then any  limit of $\left(\mu_n^{F_n}\right)_{n\in \mathfrak n}$  is a $T$-invariant Borel probability measure. 
\end{lem}

\begin{proof}
Let $\mathfrak n'$ be a subsequence of $\mathfrak n$ such that $\left(\mu_n^{F_n}\right)_{n\in \mathfrak n'}$ is converging to some $\mu'$.  It is enough to check that 
$\left|\int \phi\, d\mu_n^{F_n}-\int \phi\circ T\, d\mu_n^{ F_n}\right|$ goes to zero  for when $\mathfrak n'\ni n\rightarrow +\infty$ for any $\phi:X\rightarrow \mathbb R$  continuous.
 
 This follows from 
\begin{align*}
\int \phi\, d\mu_n^{F_n}-\int \phi\circ T\, d\mu_n^{ F_n}&=\frac{1}{\sharp F_n} \int \left(\sum_{\stackrel{k+1\in F_n}{k\notin F_n}}\phi\circ T^{k}- \sum_{\stackrel{k+1\notin F_n}{k\in F_n}} \phi\circ T^k\right) d\, \mu_n,\\
\left|\int \phi\, d\mu_n^{F_n}-\int \phi\circ T\, d\mu_n^{F_n}\right| & \leq \sup_{x\in X}|\phi(x)| \frac{\sharp\partial F_n}{\sharp F_n},\\
\liminf_{n\in \mathfrak n} \left|\int \phi\, d\mu_n^{F_n}-\int \phi\circ T\, d\mu_n^{F_n}\right| &\leq  \sup_{x\in X}|\phi(x)| \liminf_{n\in \mathfrak n}\frac{\sharp\partial F_n}{\sharp F_n}, \\
&\leq \sup_{x\in X}|\phi(x)|  \frac{d^{\mathfrak n}(\partial \mathcal F) }{ \underline{d}^{\mathfrak n}(\mathcal F)}= 0.
\end{align*}

 \end{proof}
 

\subsection{Subadditive cocycles}
 We fix  a general continuous subadditive process $\Phi=(\phi_n)_{n\in \mathbb N}$ with respect to $(X,T)$, i.e. $\phi_0=0$, $\phi_n:X\rightarrow \mathbb R$ is a continuous function for all $n$ and $\phi_{n+m}\leq \phi_n+\phi_m\circ T^n$ for all $m,n$. In the proof of the main theorem we will only  consider additive cocycles, but we think it could be interesting to consider general subadditive cocycles in other contexts.
 
  Observe that $\Phi^+=(\phi_n^+)_n$, with $\phi_n^+=\max(\phi_n,0)$ for all $n$, is also subadditive. For any $\mu \in \mathcal M(X,T)$, we let $\phi^+(\mu)=\lim_n\frac{1}{n}\int \phi_n^+\, d\mu=\inf_n\frac{1}{n}\int \phi_n^+\, d\mu$ (the existence of the limit follows from the subadditivity property). 
 Recall also that by the subadditive ergodic theorem \cite{King}, the limit $\phi_*(x)=\lim_n\frac{\phi_n(x)}{n}$ exists for $x$ in a set of full measure with respect to any invariant measure $\mu$. When $\phi_*(x)\geq 0$ for $\mu$-almost every point  $x$ with $\mu\in \mathcal M(X,T)$, then we have $\phi(\mu):=\int \phi_*(x)\, d\mu(x)=\phi+(\mu)$. If $\mu$ is moreover ergodic, then $\phi_*(x)=\phi^+(\mu)$ for $\mu$ almost every $x$.

Let  $E:Y\rightarrow\mathcal P_\N$ be a measurable sequence of integers defined on a Borel subset $Y$ of $X$. 
For a set  $F_n\in \mathcal P_n$ with $\partial F_n\subset E(x)$ for some $x\in X$, we may  write $F_n^-$ uniquely as  the finite union of $E(x)$-irreducible   intervals $F_n^-=\bigcup_{\mathsf k\in \mathsf K} \G a_\mathsf k, b_\mathsf k\G$.  Let $n_\mathsf k=b_\mathsf k-a_\mathsf k$ for any $\mathsf k\in \mathsf K$. Then we define 
$$\forall x\in X, \ \phi_E^{F_n}(x):=\sum_{\mathsf k\in \mathsf K}\phi_{n_\mathsf k}(T^{a_\mathsf k}x)$$
When $\Phi$ is additive, i.e. $\Phi_n=\sum_{0\leq k<n}\phi\circ T^k$ for some continuous function $\phi:X\rightarrow \mathbb R$, we always have   $\phi_E^{F_n}(x)= \phi^{F_n}(x):=\sum_{k\in F_n^-}\phi(T^kx)$. 

The set valued map $E$ is said \textbf{$a$-large} with respect to  $\Phi$ for some $a\geq 0$ when we have $\phi_{l-k}(T^kx)\geq (l-k)a$ for all consecutive integers $l> k$ in $E(x)$. 

\begin{lem}\label{cocycle} Let $\Phi$, $F_n$ and $E$ as above. Assume $E$ is $0$-large.
Then for all $x\in X $ and for all integers $n\geq N\geq M$ we have: 

$$\frac{\phi_E^{F_n}(x)}{\sharp F_n}\geq \int \frac{\phi_N^+}{N} \, d\delta_x^{F_n}-\frac{d_n\left (F_n\setminus E_M(x)\right)+Nd_n(\partial F_n)+4M/N}{d_n(F_n)}\sup_y|\phi_1(y)|$$
\end{lem}

\begin{proof}
Let $k\in \{0,\cdots, N-1\}$ and $l\in \N$. The interval of integers $J_{k,l}=\G k+lN,k+(l+1)N\G$ may be written as 
$$J_{k,l}=I_1\coprod I_2\coprod I_3\coprod I_4$$ where $I_1$ is the union of  disjoint $E$-irreducible intervals of length less than $M$ contained in $J_{k,l}$, $I_2 \subset \N\setminus F_n^-$, $I_3\subset F_n^-\setminus E_M(x)$ and $I_4$ is the union of at most two subintervals of $E$-irreducible intervals of length less than $M$ containing an extremal point of $J_{k,l}$.  
 
 Therefore for a fixed $k$, by summing over all $l$ with $k+lN\in F_n$ we get as $E$ is $0$-large and $\Phi$ is subadditive:

\begin{align*}
&\sum_{l, \ k+lN\in F_n }\phi_N^+(T^kx) \\
&  \leq \sum_{\mathsf k\in \mathsf K, \ \G a_\mathsf k, b_\mathsf k\G \cap (k+N\N)=\emptyset }\phi_{n_\mathsf k}(T^{a_\mathsf k}x)+\sup_y |\phi_1(y)| \left(N\sharp \partial F_n+\sharp \left(F_n\setminus E_M(x)\right)+2M([n/N]+1)\right) \\
& \leq
\phi_E^{F_n}(x) +\sup_y |\phi_1(y)| \left(N\sharp \partial F_n+\sharp \left( F_n\setminus E_M(x)\right)+2M([n/N]+1)\right).
\end{align*}
Then by summing over all $k\in \{0,\cdots, N-1\}$ and dividing by $N$, we conclude that 
$$\sharp F_n\int \frac{\phi_N^+}{N} \, d\delta^{F_n}_x\leq  \phi_E^{F_n}(x) +\sup_y |\phi_1(y)| \left(N\sharp \partial F_n+\sharp F_n\setminus E_M(x)+2M([n/N]+1)\right).$$

\end{proof}

\subsection{Positive exponent of empirical measures for additive cocycles}

We consider here an additive  cocycle $\Phi$ associated to a continuous  function $\phi:X\rightarrow \mathbb R$. With the notations of Lemma \ref{measurable} and Lemma \ref{fol}, we have :

\begin{lem}\label{large}
Let  $(\mu_n)_{n\in \mathfrak n}$ with $\mu_n(A_n)=1$ for all $n\in \mathfrak n$. 
Assume  $E$  is $a$-large with $a>0$.  Then  for any weak-$*$ limit $\mu$ of $\mu_n^{F_n}$ we have 
$$\phi_*(x)\geq a\text{ for } \mu \text{ a.e. }x. $$
\end{lem}

\begin{proof}
We claim that for any $0<\alpha<1$ and  any $\epsilon>0$, there is arbitrarily large  $N_0$ such that  \begin{eqnarray}\label{zut}\limsup_n\mu_n^{F_n}(\phi_{N_0}/N_0\geq \alpha a)\geq 1-\epsilon.
\end{eqnarray}
By weak-$*$ convergence of $\mu_n$ to $\mu$, it will imply, the set $\{\phi_{N_0}/N_0\geq \alpha a\}$ being closed :
$$\mu(\phi_{N_0}/N_0\geq \alpha a)\geq 1-\epsilon.$$
Then we may consider a sequence $(N_k)_k$ going to infinity such that 
$$\mu(\phi_{N_k}/N_k\geq \alpha a)\geq 1-\epsilon/2^k.$$
Therefore $\mu\left(\bigcap_k\{\phi_{N_k}/N_k\geq \alpha a\}\right)\geq 1-2\epsilon$. We conclude $\limsup_n\frac{\phi_n(x)}{n}\geq \alpha a$ for $\mu$ a.e. $x$ by letting $\epsilon$ go to zero.

Let us show now our first claim  (\ref{zut}). It is enough to show the inequality  for $\mu_n=\delta_x$ uniformly in $x\in A_n$. We use the same notations as in the  proof of Lemma \ref{cocycle}.  Fix $x\in A_n$. For $k,l$ with $k+lN\in F_n$, the interval $J_{k,l}
$ is said admissible, when $\phi_{N}(f^{k+lN}x)/N\geq \alpha a$. 
If $J_{k,l}$ is not admissible we have 
\begin{align*}
\phi_{N}(f^{k+lN}x)&\geq \sum_{i\in I_1}\phi(f^ix)-  \sup_{y}|\phi(y)| \sharp (I_2\cup I_3\cup I_4), \\
\alpha a N &\geq\ a\sharp I_1-  \sup_{y}|\phi(y)| \sharp (I_2\cup I_3\cup I_4),\\
&\geq aN- (a+\sup_{y}|\phi(y)|) \sharp (I_2\cup I_3\cup I_4), \\
\sharp  (I_2\cup I_3\cup I_4) & \geq \frac{(1-\alpha)aN}{a+\sup_{y}|\phi(y)|}.
\end{align*}
 If we sum over all $l$ with $k+lN\in F_n$ and then over $k\in \{0,\cdots, N-1\}$, we get by arguing as in the  proof of Lemma \ref{cocycle} :
\begin{align*} 
 N \Big( N\sharp \partial F_n+\sharp\left(F_n\setminus E_M(x)\right)+2M([n/N]+1)\Big)\geq \\\sharp\{J_{k,l}\text{ not admissible}, \ k+lN\in F_n\}\times \frac{(1-\alpha)aN}{a+\sup_{y}|\phi(y)|}.
 \end{align*}
Therefore by Lemma \ref{measurable} (third and fourth items) we have  for $\mathfrak n \ni n\gg N \gg M$ uniformly in $x\in A_n$,
$$\sharp\{J_{k,l} \text{ not admissible}, \ k+lN\in F_n\}\leq \epsilon \sharp F_n.$$
By definition of admissible intervals we conclude that 
$$\limsup_n\delta_x^{F_n}(\phi_{N}/N\geq \alpha a)\geq 1-\epsilon.$$
\end{proof}

\subsection{Entropy of empirical measures}
Following Misiurewicz's proof of the variational principle, we estimate the entropy of empirical measures from below. For a finite partition $P$ of $X$ and  a finite subset $F$ of $\N$, we let $P^F$ be the iterated partition $P^F=\bigvee_{k\in F}f^{-k}P$. When $F=\G 0,n-1\R$, $n\in \N$, we just let $P^{F}=P^n$. We denote by $P(x)$ the element of $P$ containing $x\in X$.

For a Borel probability measure $\mu$ on $X$, the static entropy $H_\mu(P)$ of $\mu$ with respect to a (finite measurable) partition $P$  is defined as follows:
\begin{align*}
H_\mu(P)&=-\sum_{A\in P}\mu(A)\log \mu(A),\\
&=-\int \log \mu\left(P(x)\right)\, d\mu(x).
\end{align*}
When $\mu$ is $T$-invariant, we recall that  the measure theoretical entropy of $\mu$ with respect to $P$ is then   $$h_\mu(P)=\lim_n\frac{1}{n}H_{\mu}(P^n)$$ 
and the entropy $h(\mu)$ of $\mu$ is 
$$h(\mu)=\sup_Ph_\mu(P).$$

We will use the  two following standard properties of the static entropy\cite{dow}:
\begin{itemize}
\item for a fixed partition $P$, the map $\mu\mapsto H_\mu(P)$ is concave on $\mathcal M(X)$, 
\item for two partitions $P$ and $Q$, the joined partition $P\vee Q$ satisfies 
\begin{align}\label{ziz}
H_{\mu}(P\vee Q)&\leq H_\mu(P)+H_\mu(Q). 
\end{align}
\end{itemize}

\begin{lem}\label{comput}Let $\mathcal F=(F_n)_{n\in \mathfrak n}$ be a F\"olner sequence with $\underline{d}^{\mathfrak n}(\mathcal F)>0$. For any   measurable finite partition $P$ and $m\in \N^*$, there exist   a sequence $(\epsilon_n)_{n\in \mathfrak n}$ converging to $0$ 
such that 
$$\forall n\in \mathfrak n, \ \frac{1}{m}H_{\mu_n^{F_n}}(P^m)\geq \frac{1}{\sharp F_n}H_{\mu_n}(P^{F_n})+\epsilon_n.$$
\end{lem}
\begin{proof}
When $F_n$ is an interval of integers, we have \cite{Mis} :
\begin{equation}\label{mi}\frac{1}{m}H_{\mu_n^{F_n}}(P^m)\geq \frac{1}{\sharp F_n}H_{\mu_n}(P^{F_n})-\frac{3m\log \sharp P}{\sharp F_n}.
\end{equation}
Consider a general set $F_n\in \mathcal P_n$.  We decompose $F_n$ into connected components  $F_n=\coprod_{k=1,\cdots, K } F_n^k$.  Observe $K\leq \sharp \partial F_n$.  Then we get  :
\begin{align*}
\frac{1}{m}H_{\mu_n^{F_n}}(P^m)&\geq \sum_{k=1}^K\frac{\sharp F_n^k}{m\sharp F_n}H_{\mu_n^{F^k_n}}(P^m), \text{by concavity of $\mu \mapsto H_{\mu}(P^m)$}\\
&\geq  \frac{1}{\sharp F_n}\sum_{k=1}^KH_{\mu_n}(P^{F_n^k})-\frac{3mK\log \sharp P}{\sharp F_n},\text{ by applying (\ref{mi}) to each $F_n^k$},\\
&\geq \frac{1}{\sharp F_n}H_{\mu_n}(P^{F_n})-3m\log \sharp P\frac{\sharp \partial F_n}{\sharp F_n}, \text{ according to  (\ref{ziz})}.
\end{align*}
This concludes the proof with $\epsilon_n=3m\frac{\sharp \partial F_n}{\sharp F_n} \log \sharp P$, because $\mathcal F$ is a F\"olner sequence with $\underline{d}^\mathfrak n(\mathcal F)>0$.
\end{proof}

With the notations of Lemma \ref{measurable} we let $\mu_n$ be the probability measure induced by $\lambda$ on $A_n$, i.e. $\mu_n=\frac{\lambda(A_n\cap \cdot)}{\lambda(A_n)}$. Let $\Psi=(\psi_n)_n$ be a continuous subadditive process such that $E$ is $0$-large with respect to $\Psi$. We assume that $\lambda$ satisfies the following \textit{F\"olner Gibbs property } with respect to the subadditive cocycle $\Psi$ :
\begin{align}\label{hyp}
&\text{There exists $\epsilon>0$ such that }  \nonumber \\
& \text{we have for any partition $P$ with  diameter less than $\epsilon$ :}\tag{$\text{G}$}\\
 & \exists N \ \forall x\in A_n \textrm{ with } N<n\in \mathfrak{n}, \  \ \  \frac{1}{\lambda\left(P^{F_n}(x)\cap A_n\right)}\geq e^{\psi_E^{F_n}(x)}. \nonumber
\end{align}

\begin{prop}\label{pour}
Under the above hypothesis (\ref{hyp}), any weak-$*$ limit $\mu$ of $(\mu_n^{F_n})_{n\in \mathfrak n}$ satisfies 
$$h(\mu)\geq \psi^+(\mu).$$
\end{prop}

\begin{proof}
Without loss of generality we may assume $(\mu_n^{F_n})_{n\in \mathfrak n}$ is converging to  $\mu$. Take  a partition $P$ with $\mu(\partial P)=0$ and with diameter less than $\epsilon$. In particular we have for all fixed $m\in \N $:
\begin{equation*}\frac{1}{m}H_{\mu}(P^m)=\lim_n\frac{1}{m}H_{\mu_n^{F_n}}(P^m).
\end{equation*}
Then  we get for $n\gg N\gg M\gg  m$  

\begin{align*}
\frac{1}{m}H_{\mu}(P^m)\geq & \limsup_{n\in\mathfrak n} \ \frac{1}{\sharp F_n}H_{\mu_n}(P^{F_n}), \textrm{ by Lemma \ref{comput}},\\
\geq &  \limsup_{n\in\mathfrak n} \frac{1}{\sharp F_n}\int \left(- \log \lambda \left( P^{F_n}(x)\cap A_n \right)+\log \lambda(A_n)\right) \, d\mu_n(x),\\
 \geq &  \limsup_{n\in\mathfrak n} \int \frac{\psi_E^{F_n}}{\sharp F_n}  \, d\mu_n(x), \textrm{ by Hypothesis (G)},\\
 \geq & \limsup_{n\in\mathfrak n}  \bigg(\int \frac{\psi_N^+}{N}\, d \mu_n^{F_n}\\
 & -\frac{\sup_{y}|\psi(y)|\left(\sup_{x\in A_n}d_n(F_n\setminus E_M(x))+Nd_n(\partial F_n) +4M/N\right)  }{d_n(F_n)}\bigg), \textrm{ by Lemma \ref{cocycle}},\\
\geq &  \int \frac{\psi_N^+}{N}\,d \mu-\frac{1}{\underline{d}^{\mathfrak n}(\mathcal F)}\left(\sup_{y}|\psi(y)|\left(\limsup_{n\in \mathfrak n}\sup_{x\in A_n}d_n(F_n\setminus E_M(x))+4M/N\right) \right),\\
\geq & \psi^{+}(\mu)-\frac{1}{\underline{d}^{\mathfrak n}(\mathcal F)}\left(\sup_{y}|\psi(y)|\left(\limsup_{n\in \mathfrak n}\sup_{x\in A_n}d_n(F_n\setminus E_M(x))+4M/N\right)\right).
\end{align*}
Letting $N$, then $M$, then $m$ go to infinity,  we conclude that 
\begin{align*}
h(\mu)\geq h_{\mu}(P)
&\geq \psi^{+}(\mu).
\end{align*}
\end{proof}

In the following we will also consider a general\footnote{The set $E$ is not assumed here to be $0$-large with respect to $\Psi$.}  additive cocycle $\Psi=(\psi_n)_n$ associated to a continuous function $\psi:X\rightarrow \mathbb R$. Then  for any F\"olner sequence $(F_n)_n$, the \textit{F\"olner Gibbs property } with respect to the  additive cocycle $\Psi$ may be simply written as follows:
\begin{align}\label{hyp}
&\text{There exists $\epsilon>0$ such that }  \nonumber \\
& \text{we have for any partition $P$ with  diameter less than $\epsilon$ :}\tag{$\text{H}$}\\
 & \exists N \ \forall x\in A_n \textrm{ with } N<n\in \mathfrak{n}, \  \ \  \frac{1}{\lambda\left(P^{F_n}(x)\cap A_n\right)}\geq e^{\psi^{F_n}(x)}. \nonumber
\end{align}

In this additive setting we get : 
\begin{prop}\label{por}
Under the above hypothesis (\ref{hyp}), any weak-$*$ limit $\mu$ of $(\mu_n^{F_n})_{n\in \mathfrak n}$ satisfies 
$$h(\mu)\geq \psi(\mu).$$
\end{prop}

\begin{proof}
Let  $P$ as in the proof of Proposition \ref{pour}. Then for $n\gg N\gg m$ we obtain by following this proof :
\begin{align*}
\frac{1}{m}H_{\mu}(P^m)&\geq   \limsup_{n\in\mathfrak n} \int \frac{\psi^{F_n}}{\sharp F_n}  \, d\mu_n(x), \textrm{ by Hypothesis (H)},\\
& \geq  \limsup_{n\in\mathfrak n}  \int \psi \, d \mu_n^{F_n},\\
 & \geq \psi(\mu).
 \end{align*}
 Letting $m$ go to infinity, we conclude that $h(\mu)\geq \psi(\mu)$.
\end{proof}

\section{Geometric times}

 Let $r\geq 2$ be an integer   and let $(M,\|\cdot\|)$ be  a $C^{r}$ smooth compact Riemannian  manifold, not necessarily a surface for the moment.  We denote by $\mathrm d$   the distance induced by the Riemannian structure on $M$. We also consider a distance  $\hat{\mathrm{d}}$ on the projective tangent bundle $ \mathbb P TM$, such that $\hat{\mathrm{d}}(\hat x, \hat y)\geq \mathrm{d}(\pi\hat x, \pi \hat y)$  for all $\hat x, \hat y \in \mathbb P TM$ with  $\pi:\mathbb P TM\rightarrow M $ being the natural projection. For a $C^r$ map $f:M\rightarrow M$ or a $C^r$ curve $\sigma:[0,1]\rightarrow M$ we may define the norm  $\|d^sf\|_\infty$ and $\|d^s\sigma\|_\infty$ for $1\leq s\leq r$ as the supremum norm of the $s$-derivative of the induced maps  through the charts of a given atlas or  through the exponential map $\exp$. In the following, to simplify the  presentation we lead the computations as $M$ was an Euclidean space.  For a $C^1$ curve $\sigma:I\rightarrow M$, $I$ being a compact interval of $\mathbb R$, we let $\sigma_*=\sigma(I)$. The length of $\sigma_*$ for the induced Riemannian metric is denoted by $|\sigma_*|$.  For a fixed curve $\sigma$ we also let   $v_x\in \mathbb PTM$ be  the line tangent to $\sigma_*$ at $x$ and we write $\hat x=(x,v_x)$.
 
 We denote by $F$ the  projective action $F:\mathbb PTM\circlearrowleft$ induced by $f$ and we consider:
the additive derivative cocycle $\Phi=(\phi_k)_k$ for $F$ on $\mathbb PTM$ given by $\phi(x,v)=\phi_1(x,v)=\log \|d_xf(v)\|$, where  we have identified the line $v_x$ with one of its unit generating vectors.

 
\subsection{Bounded curve} \label{curvee}
Following \cite{bure} a $C^r$  smooth curve  $\gamma:[-1,1]\rightarrow  M$   is said \textbf{bounded }when   \begin{equation*}
\max_{s=2,\cdots, r}\|d^s\gamma\|_\infty\leq \frac{1}{6}\|d\gamma\|_\infty. \end{equation*}
We first recall some basic properties of bounded curves (see  Lemma 7 in \cite{bure}).
A bounded curve has bounded distorsion meaning that 
\begin{equation}\label{dist}\forall t,s\in [-1,1], \ \frac{\|d\gamma(t)\|}{\|d\gamma(s)\|}\leq 3/2. \end{equation} 
Indeed, if $t_*\in [-1,1]$ satisfies $\|d\gamma(t_*)\|=\|d\gamma\|_\infty$ then we have for all $s\in [-1,1]$,
\begin{align*}
\|d\gamma(t_*)-d\gamma(s)\|&\leq 2 \|d^2\gamma\|_{\infty},\\
&\leq \frac{1}{3}\|d\gamma(t_*)\|, \\
\text{ therefore \ }\frac{2}{3}\|d\gamma(t_*)\|&\leq \|d\gamma(s)\| \leq \|d\gamma(t_*)\|, 
\end{align*}
The projective component of $\gamma$ oscillates also slowly. If we identify  $M$ with $\mathbb R^2$ \footnote{This will be always possible as we will only consider  curves with diameter less than the radius of injectivity.}, we have 
\begin{align}\label{oscill}\|d\gamma(t_*)\|\cdot \sin \angle d\gamma(t_*), d\gamma(s)&\leq \|d\gamma(t_*)-d\gamma(s)\|\leq  \frac{1}{3}\|d\gamma(t_*)\|,\nonumber \\
 \angle d\gamma(t_*), d\gamma(s)&\leq \pi/6.
\end{align}

When moreover $\|d\gamma\|_\infty\leq \epsilon$ we say that $\gamma$ is \textbf{strongly $\epsilon$-bounded}. In particular such a map satisfies  $\|\gamma\|_r:=\max_{1\leq s\leq r}\|d^s\gamma\|\leq \epsilon$, which is the standard  $C^r$ upper bound required for the reparametrizations in the usual Yomdin's theory. But this last condition does not  allow to control the distorsion along the curve in general.\\ 

If $\gamma$ is bounded then so is $\gamma_a=\gamma(a\cdot ):[-1,1]\rightarrow M$ for any $a\leq\frac{2}{3}$:
\begin{align*}
\forall s\geq 2, \ \|d^s\gamma_a\|_\infty&\leq \frac{1}{6} a^s\|d\gamma\|_\infty, \\
&\leq  \frac{1}{6} a^s\frac{3}{2}\|d\gamma(0)\|,\\
&\leq \frac{1}{6}a^{s-1}\|d\gamma(0)\|,\\
&\leq \frac{1}{6}\|d\gamma_a\|_\infty.
\end{align*}
As $\|d\gamma_a\|_\infty\leq a\|d\gamma\|_\infty$, if $\gamma$ is moreover strongly $\epsilon$-bounded, then  $\gamma_a$ is $a\epsilon$-strongly bounded. 

\begin{lem}\label{tech}Let $\gamma:[-1,1]\rightarrow M$ be a $C^r$ bounded curve with $\|d\gamma\|_\infty\geq \epsilon$. Then there is a family of affine maps $\iota_j:[-1,1]\circlearrowleft$, $j\in L:=\underline{L}\cup \overline{L}$ such that:
\begin{itemize}
\item each $\gamma \circ \iota_j$ is $\epsilon$-bounded and $\|d(\gamma\circ \iota_j)(0)\|\geq \frac{\epsilon}{6}$,
\item $[-1,1]$ is the union of $\bigcup_{j\in \underline{L}}\iota_j([-1,1])$ and $\bigcup_{j\in \overline{L}}\iota_j([-\frac{1}{3},\frac{1}{3}])$, 
\item $\sharp \underline{L}\leq 2$ and $\sharp \overline{L}\leq 6\left( \frac{\|d\gamma\|_\infty}{\epsilon}+1\right)$, 
\item for any  $x\in \gamma_*$, we have 
$\sharp \{j\in L, \ (\gamma \circ \iota_j)_*\cap B(x,\epsilon)\neq \emptyset\}\leq 100.$
\end{itemize}
\end{lem}

\begin{proof}[Sketch of proof]
For the first three items it is enough to consider affine reparametrizations of $[-1,1]$ with rate $\frac{2\epsilon}{3\|d\gamma\|_\infty}$.  As the bounded map $\gamma$ stays in a cone of opening angle $\pi/6$, its intersection with $B(x,\epsilon)$ has length less than $2\epsilon$. The last item follows then easily.
\end{proof}
Fix a $C^r$ smooth diffeomorphism $f:M\circlearrowleft$. 
A curve $\gamma:[-1,1]\rightarrow M$   is  said  \textbf{$n$-bounded} (resp. \textbf{strongly $(n,\epsilon)$-bounded}) when $f^k\circ \gamma$ is  bounded (resp. strongly $\epsilon$-bounded) for $k=0,\cdots, n$.
A strongly $\epsilon$-bounded curve $\gamma$ is contained in the dynamical ball $B_n(x,\epsilon):=\{y\in M, \ \forall k=0,\cdots, n-1, \  \mathrm{d}(f^kx,y )<\epsilon \}$ with $x=\gamma(0)$.

Fix  a   $C^r$ curve $\sigma:I\rightarrow M$.   For $x\in \sigma_*$, a positive  integer $n$ is called an  \textbf{$(\alpha, \epsilon)$-geometric time} of $x$ when there exists an affine map 
  $\theta_n:[-1,1]\rightarrow I$  such that $\gamma_n:=\sigma\circ \theta_n$   is strongly $(n,\epsilon)$-bounded, $\gamma_n(0)=x$ and $\|d(f^n\circ \gamma_n)(0)\|\geq \frac{3}{2}\alpha\epsilon$. 
The concept of  \textbf{$(\alpha, \epsilon)$-geometric time} is  almost independent of $\epsilon$. 
  Indeed it follows from the above observations that, if $n$ is a $(\alpha, \epsilon)$-geometric time of $x$, then it is also a $(\alpha, \epsilon')$-geometric time for $\epsilon'<\frac{2\epsilon}{3}$. Moreover if $n$ is a $(\alpha, \epsilon)$-geometric time of $x$ with $\gamma_n$ the  associated curve, then $n$ is a $(\frac{2}{3}\alpha,\frac{2}{3}\epsilon)$-geometric  time of  any $y\in \gamma_n([1/3,1/3])$: if $y=\gamma_n(t)$ for $t\in[-1/3,1/3]$ then $\tilde \gamma_n:=\gamma_n(t+\frac{2}{3}\cdot)$ is strongly  $(n,\frac{2}{3}\epsilon)$-bounded and satisfies  $\tilde \gamma_n(0)=y$ and $\|d(f^n\circ \tilde\gamma_n)(0)\|=\frac{2}{3}\|d(f^n\circ\gamma_n)(t)\|\geq \frac{4}{9}\|d(f^n\circ \gamma_n)(0)\|\geq \frac{2}{3}\alpha \epsilon$.\\

   We let $D_n(x)$ and $H_n(x)$ be  the images of $f^n\circ \gamma_n$ and $\gamma_n$ respectively with $\gamma_n$ as above of  maximal length. We define the semi-length of $D_n(x)$ as the minimum of the lengths of $f^n\circ \gamma_n([0,1])$ and $f^n\circ \gamma_n([-1,0])$.  The semi-length of $D_n(x)$ is larger than $\alpha\epsilon$ at a $(\alpha, \epsilon)$-geometric time $n$. One can also easily checks that  the curvature of $f^n\circ \sigma$ at $f^nx$ is bounded from above by $\frac{1}{\alpha\epsilon}$. From the bounded distorsion property of bounded curves (\ref{dist}) we get 
\begin{equation}\label{distor}\forall y,z\in H_n(x) \, \forall 0\leq l<n,\ \ \    \frac{e^{\phi_{n-l}(F^l\hat y)}}{e^{\phi_{n-l}(F^l\hat z)}}\leq \frac{9}{4}.\end{equation}

\subsection{Reparametrization Lemma }
We consider a $C^r$ smooth  diffeomorphism $g:M\circlearrowleft$ and a $C^r$ smooth curve $\sigma:I\rightarrow M$ with $\mathbb N \ni r\geq 2$. 
We state a global reparametrization lemma to describe the dynamics on $\sigma_*$. We will apply this lemma to $g=f^p$ for large $p$ with $f$ being the $C^r$ smooth system under study. We denote by $G$ the map induced by $g$ on $\mathbb P TM$.\\

We will encode the dynamics of $g$ on $\sigma_*$ with a tree, in a similar  way the symbolic dynamic associated to monotone branches encodes the dynamic of a continuous piecewise monotone interval map.   A  weighted directed rooted tree $\mathcal T$ is a  directed rooted tree whose  edges are labelled. Here the weights on the  edges are pairs of  integers. Moreover the nodes of our tree will be coloured,  either  in blue or in red. \\

 We let $\mathcal T_n$ (resp. $\underline{\mathcal T_n}$, $\overline{\mathcal T_n}$) be the set of nodes  (resp. blue, red nodes) of level $n$. For all $k\leq n-1$  and for all $\mathbf i^n\in \mathcal T_n$, we also let  $\mathbf i^n_{k}$ be the node of level $k$ leading to $\mathbf i^{n}$.  For   $\mathbf i^n\in \mathcal T_n$, we let $k(\mathbf i^n)=(k_1(\mathbf i^n),k'_1(\mathbf i^n), k_2(\mathbf i^n) \cdots,k'_n(\mathbf i^n) )$ be the $2n$-uple of  integers given  by the sequence of labels along the path from the root $\mathbf i^0$ to $\mathbf i^n$,  where  $\left(k_l(\mathbf i^n),k'_l(\mathbf i^n)\right)$ denotes the label of the edge joining $\mathbf i^n_{l-1}$ and $\mathbf i^n_{l}$.\\

For $x\in \sigma_*$, we let $k(x)\geq k'(\hat x)$ be the following integers:
$$k(x):=\left[\log\|d_{x}g \|\right], $$
$$k'( \hat x):=\left[\log \|d_xg(v_x)\|\right].$$
Then for all $n\in \mathbb N^*$ we  define 
$$k^n( x)=(k( x),k'(\hat x), k(gx),\cdots k'(G^{n-2}\hat x), k(g^{n-1}x), k'(G^{n-1}\hat x) ).$$
For a $2n$-uple of   integers $\mathbf k^n=(k_1,k'_1, \cdots k'_n,k_n)$ we consider then
$$\mathcal H(\mathbf k^n):=\left\{x\in \sigma_*, \,  k^n( x)=\mathbf k^n\right\}.$$

We restate the Reparametrization Lemma  (RL for short) proved in \cite{bure} in a \textit{global} version. Let $\exp_x$ be the exponential map at $x$ and let $R_{inj}$ be the radius of injectivity of $(M, \|\cdot\|)$.
\begin{Rep}
 Let $\frac{R_{inj}}{2}>\epsilon>0$
 satisfying  $\|d^sg_{2\epsilon}^x\|_\infty\leq 3\epsilon \|d_xg\|$ for all $s=1,\cdots,r$ and all $x\in M$, where $g^x_{2\epsilon}= g\circ  \exp_x(2\epsilon\cdot): \{w_x\in T_xM, \ \|w_x\|\leq 1\}\rightarrow M$  and let  $\sigma:[-1,1]\rightarrow M$ be a strongly $\epsilon$-bounded  curve. 
 
 Then there is $\mathcal T$, a bicoloured weighted directed rooted tree,  and $\left(\theta_{\mathbf{i}^n}\right)_{\mathbf{i}^n \in \mathcal T_n}$, $n\in \mathbb N$, families of affine reparametrizations of $[-1,1]$,  such that for some universal constant $C_r$ depending only on $r$:

\begin{enumerate}
\item $\forall \mathbf i^n\in \mathcal T_n$, the curve $\sigma\circ \theta_{\mathbf i^n}$ is $(n,\epsilon)$-bounded, 
\item $\forall \mathbf i^n\in \mathcal T_n$, the affine map  $\theta_{\mathbf{i}^n}$  may be written as  $\theta_{\mathbf{i}^n_{n-1}}\circ \phi_{\mathbf i^n}$ with $\phi_{\mathbf i^n}$ being an affine contraction with rate smaller than $1/100$ and $\theta_{\mathbf{i}^n}([-1,1])\subset \theta_{\mathbf{i}^n_{n-1}}([-1/3,1/3])$ when $\mathbf{i}^n_{n-1}$ belongs to $ \overline{\mathcal T_{n-1}}$,
\item $\forall \mathbf i^n\in \overline{\mathcal T_n}$, we have $\left\|d\left(f^n\circ \sigma \circ \theta_{\mathbf{i}^n}\right)(0)\right\|\geq \epsilon/6$, 
\item  
 $\forall \mathbf k^n\in (\mathbb Z\times \mathbb Z)^n$, the  set $\sigma^{-1}\mathcal H(\mathbf k^n) $ is contained in the union of \\ $$\displaystyle{\bigcup_{\stackrel{\mathbf i^n\in \overline{\mathcal T_n}}{ k(\mathbf i^n)=\mathbf k^n}} \theta_{\mathbf{i}^n}([-1/3,1/3])} \text{ and } \displaystyle{\bigcup_{\stackrel{\mathbf i^n\in \underline{\mathcal T_n}}{ k(\mathbf i^n)=\mathbf k^n}} \theta_{\mathbf{i}^n}([-1,1])}.$$\\
  Moreover any term of these unions have a non-empty intersection with $\sigma^{-1}\mathcal H(\mathbf k^n) $,
\item $\forall  \mathbf i^{n-1}\in \mathcal T_{n-1}$ and $(k_n,k'_n)\in \mathbb Z\times \mathbb Z$ we have
$$\sharp \left\{ \mathbf i^n\in \overline{\mathcal T_n}, \ \mathbf i^n_{n-1}= \mathbf i^{n-1} \text{ and }   (k_n(\mathbf i^n),k'_n(\mathbf i^n)) =(k_n,k'_n)\right\}\leq C_re^{\max\left(k'_n,\frac{k_n-k'_n}{r-1}\right)}, $$
$$\sharp \left\{ \mathbf i^n\in \underline{\mathcal T_n}, \  \mathbf i^n_{n-1}= \mathbf i^{n-1} \text{ and }  (k_n(\mathbf i^n),k'_n(\mathbf i^n)) =(k_n,k'_n)\right\}\leq C_r e^{\frac{k_n-k'_n}{r-1}}.$$
\end{enumerate}
\end{Rep}

\begin{proof}
We argue  by induction on $n$. For $n=0$ we let $\mathcal T_0=\underline{\mathcal T_0}=\{\mathbf i^0\}$ and we just take $\theta_{\mathbf i^0}$ equal to  the identity map on $[-1,1]$. Assume the tree and the associated reparametrizations have  been built till the level $n$.

Fix $\mathbf i^n\in\mathcal T_n$ and let \begin{align*}\hat \theta_{\mathbf i^n}:=&\left\{
    \begin{array}{ll}
       \theta_{\mathbf i^n}(\frac{1}{3}\cdot ) & \mbox{if }  \mathbf i^n\in \overline{\mathcal  T_n},\\
     \theta_{\mathbf i^n} & \mbox{if }  \mathbf i^n\in \underline{\mathcal  T_n}.
    \end{array}\right.
    \end{align*}
 We will define the children $
\mathbf i^{n+1}$ of $\mathbf i^n$, i.e. the nodes $\mathbf i^{n
+1}\in \mathcal T_{n+1}$ with $\mathbf i_n^{n+1}=\mathbf i^{n}
$. The label on the edge joining $\mathbf i^n$ to $\mathbf i^{n+1}
$ is a pair $(k_{n+1}, k'_{n+1})$ such that the $2(n+1)$-uple $
\mathbf{k}^{n+1}=(k_1(\mathbf i^n), \cdots, k'_n(\mathbf 
i^n),k_{n+1},k'_{n+1})$ satisfies $\mathcal H(\mathbf k^{n+1})\cap \left(\sigma \circ \hat\theta_{\mathbf i^n}\right)_*\neq \emptyset$. We fix such a pair $(k_{n+1}, k'_{n+1})$ and the associated sequence $\mathbf{k}^{n+1}$. We let $\eta,\psi:[-1,1]\rightarrow M$ be the curves defined as: \begin{align*}
\eta:=&\sigma\circ \hat \theta_{\mathbf i^n},\\
 \psi:=&g^n\circ \eta.\\ \end{align*}


\underline{\textit{First step :}} \textbf{Taylor polynomial approximation.} One   computes for an affine map $\theta:[-1,1]\circlearrowleft$ with contraction rate $b$ precised later and   with $y= \psi(t)\in g^{n}\mathcal H(\mathbf k^{n+1}) $, $t\in \theta([-1,1])$:

\begin{align*}\|d^r(g\circ \psi\circ \theta)\|_\infty &\leq b^r \left\|d^{r}\left(g_{2\epsilon}^y \circ \psi_{2\epsilon}^y\right)\right\|_\infty , \textrm{with $\psi_{2\epsilon}^y:=(2\epsilon)^{-1}\exp_y^{-1}\circ \psi$,}\\
&\leq  b^r\left\|d^{r-1}\left( d_{\psi_{2\epsilon}^y}g_{2\epsilon}^y\circ d\psi_{2\epsilon}^y \right) \right\|_\infty,\\
&\leq b^r 2^r  \max_{s=0,\cdots,r-1}\left\|d^s\left(d_{\psi_{2\epsilon}^y}g_{2\epsilon}^y\right)\right\|_{\infty}\|\psi_{2\epsilon}^y \|_r.
\end{align*}
By assumption on $\epsilon$, we have $\|d^s g_{2\epsilon}^y\|_{\infty}\leq 3\epsilon\|d_y g\|$ for any $r\geq s\geq 1$.
Moreover  $\|\psi_{2\epsilon}^y \|_r\leq (2\epsilon)^{-1}\|d\psi\|_\infty\leq 1$ as $\psi$ is strongly $\epsilon$-bounded. 
Therefore by  Fa\'a di Bruno's formula, we get  for some\footnote{Although these constants may differ at each step, they are all denoted by $C_r$.}  constants $C_r>0$ depending only on $r$:  
\begin{align*}\max_{s=0,\cdots,r-1}\|d^s\left(d_{\psi_{2\epsilon}^y}g_{2\epsilon}^y\right)\|_{\infty} &\leq \epsilon C_r\|d_y g\|,\\
\text{then }&,\\
\|d^r(g\circ \psi\circ \theta)\|_\infty &\leq \epsilon C_rb^r \|d_y g\|\|\psi_{2\epsilon}^y \|_r,\\
&\leq  C_rb^r \|d_y g\|\|d\psi \|_\infty, \\
&\leq ( C_r b^{r-1}\|d_yg\|) \|d(\psi \circ \theta)\|_{\infty}, \\
&\leq (C_r b^{r-1}e^{k_{n+1}}) \|d(\psi \circ \theta)\|_{\infty},  \textrm{ because $y$  belongs to $g^n\mathcal H(\mathbf k^{n+1})$}, \\
& \leq e^{k'_{n+1}-4}\|d(\psi\circ \theta)\|_\infty, \textrm{ by taking  $b=\left(C_re^{k_{n+1}-k'_{n+1}+4 }\right)^{-\frac{1}{r-1}}$.}
\end{align*}

The Taylor polynomial  $P$ at $0$ of degree $r-1$  of $d(g\circ \psi\circ \theta)$  satisfies on $[-1,1]$:
\begin{align*}
\|P-d(g\circ \psi\circ \theta)\|_{\infty}&\leq e^{k'_{n+1}-4}\|d(\psi\circ \theta)\|_\infty.
\end{align*}
We may cover $[-1,1]$ by at most $b^{-1}+1$ such affine maps $\theta$. \\

\underline{\textit{Second step :}} \textbf{Bezout theorem.}
Let $a_n:=e^{k'_{n+1}}\|d(\psi\circ \theta)\|_\infty$. Note that for $s\in [-1,1]$ with $\eta\circ \theta(s)\in \mathcal H(\mathbf k^{n+1})$
we have $\|d(g\circ \psi\circ \theta)(s)\|\in [a_ne^{-2},a_ne^{2}]$, therefore $\|P(s)\|\in [a_ne^{-3},a_ne^3]$. Moreover if we have now  $\|P(s)\|\in [a_ne^{-3},a_ne^3]$ for some $s\in [-1,1]$ we  get also $\|d(g\circ \psi\circ \theta)(s)\|\in [a_ne^{-4},a_ne^{4}]$.

 By Bezout theorem the semi-algebraic set $\{ s\in [-1,1],\  \|P(s)\|\in  [e^{-3}a_n, e^{3}a_n]\}$ is the disjoint  union of closed   intervals $(J_i)_{i\in I}$ 
with $\sharp I$ depending only on $r$. Let $\theta_i$ be the composition of $\theta$ with an affine  reparametrization from $[-1,1]$ onto $J_i$. \\

\underline{\textit{Third step :}} \textbf{ Landau-Kolmogorov inequality.}
By the Landau-Kolmogorov inequality on the interval  (see Lemma 6 in  \cite{bur}), we have for some   constants $C_r\in \mathbb N^*$  and for all $1\leq s\leq r$:
\begin{align*}
\|d^s(g\circ \psi\circ \theta_i)\|_\infty & \leq  C_r\left(\|d^r(g\circ \psi\circ \theta_i)\|_\infty +\|d(g\circ \psi\circ \theta_i)\|_\infty\right),\\
&\leq C_r\frac{|J_i|}{2}\left( \|d^r(g\circ \psi\circ \theta)\|_\infty+ \sup_{t\in J_i}\|d(g\circ \psi\circ \theta)(t)\| \right),\\
&\leq C_r a_n\frac{|J_i|}{2}.
\end{align*}
We cut again each $J_i$ into $1000C_r$ intervals $\tilde{J_i}$ of the same length with $(\eta \circ \theta)(\tilde{J}_i)\cap \mathcal H(\mathbf k^{n+1})\neq \emptyset$. Let $\tilde{\theta_i}$ be the affine reparametrization   from $[-1,1]$ onto  $\theta(\tilde{J_i})$. We check that $g\circ \psi\circ \tilde{\theta_i}$ is bounded:
\begin{align*}
\forall s=2,\cdots, r, \   \|d^s(g\circ \psi\circ \tilde{\theta_i})\|_\infty & \leq (1000C_r)^{-2} \|d^s(g\circ \psi\circ \theta_i)\|_\infty,\\
&\leq \frac{1}{6}(1000C_r)^{-1}\frac{|J_i|}{2}a_ne^{-4},\\
&\leq  \frac{1}{6}(1000C_r)^{-1}\frac{|J_i|}{2}\min_{s\in J_i}\|d(g\circ \psi\circ \theta)(s)\|,\\
&\leq  \frac{1}{6}(1000C_r)^{-1}\frac{|J_i|}{2}\min_{s\in \tilde{J}_i}\|d(g\circ \psi\circ \theta)(s)\|,\\
&\leq \frac{1}{6} \|d(g\circ \psi\circ \tilde{\theta_i})\|_\infty.
\end{align*}

\underline{\textit{ Last step :}} \textbf{$\epsilon$-bounded curve.}
Either $g\circ \psi\circ \tilde{\theta_i}$ is $\epsilon$-bounded and $\hat \theta_{\mathbf i^n}\circ \tilde{\theta_i}=\theta_{\mathbf i^{n+1}}$ for some $\mathbf i^{n+1}\in \underline{\mathcal T}_{n+1}$. Or we apply Lemma \ref{tech} to $g\circ \psi\circ \tilde{\theta_i}$ : the new affine parametrizations $\hat \theta_{\mathbf i^n}\circ \tilde{\theta_i}\circ \iota_j$, $j\in \underline L$ (resp. $j\in \overline{L}$) then  define  $\theta_{\mathbf i^{n+1}}$ for a  node $\mathbf i^{n+1}$ in $\underline{\mathcal T}_{n+1}$ (resp. $\overline{\mathcal T}_{n+1}$). Note finally that:
\begin{align*}
\sharp  \overline{L} &\leq 6\left( \frac{\|d(g\circ \psi\circ \tilde{\theta_i})\|_\infty}{\epsilon}+1\right),\\
& \leq 100 \max(e^{k'_{n+1}}b,1), \text{ as $\psi$ is $\epsilon$-bounded and $\|d\tilde{\theta_i}\|_\infty\leq b$},\\
&\leq C_r \max\left(\frac{e^{k'_{n+1}}}{e^{\frac{k_{n+1}-k'_{n+1}}{r-1}}},1\right),
\end{align*}
therefore  
\begin{align*}
\sharp \left\{ \mathbf i^{n+1}\in \overline{\mathcal T_{n+1}} \ \Big| \ \overset{\mathbf i^{n+1}_{n}= \mathbf i^{n} \text{ and }}{ (k_{n+1}(\mathbf i^{n+1}),k'_{n+1}(\mathbf i^{n+1})) =(k_{n+1},k'_{n+1})}\right\}&\leq \sum_{\tilde\theta_i}C_r \max\left(\frac{e^{k'_{n+1}}}{e^{\frac{k_{n+1}-k'_{n+1}}{r-1}}},1\right), \\
&\leq C_re^{\max\left(k'_{n+1},\frac{k_{n+1}-k'_{n+1}}{r-1}\right)}. 
\end{align*}
\end{proof}

As a corollary  of the proof of  RL  we state a \textit{local} reparametrization lemma, i.e. we only reparametrize the intersection of $\sigma_*$ with some given dynamical ball.  For $x\in \sigma_*$, $n\in \mathbb N$ and $\epsilon>0$ we let 
$$B^G_\sigma(x,\epsilon,n):=\left\{y\in \sigma_*, \ \forall k=0,\cdots, n-1, \ \hat{\mathrm{d}}(G^k\hat x,G^k\hat y)<\epsilon  \right\}.$$
For all $(x,v)\in \mathbb PTM$, we also let  $w(x,v)=w_g(x,v) :=\log\|d_{x}g\|-\log \|d_xg(v)\|$ and for all $n\in \mathbb N$ we let $w^n(x,v)=w^n_g(x,v):=\sum_{k=0}^{n-1}w(G^k(x,v))$. We consider $\epsilon>0$  as in the Reparametrization Lemma. We assume moreover that 
$$ [\hat{\mathrm{d}}((x,v), (y,w) )<\epsilon]\Rightarrow [ \left|\log \|d_xg(v)\|-\log \|d_yg(w)\|\right|<1 \text{ and }\left|\log \|d_xg\|-\log \|d_yg\|\right|<1].$$

\begin{coro}\label{local}
For any  strongly $\epsilon$-bounded curve $\sigma:[-1,1]\rightarrow M$ and for any $x\in \sigma_*$,  we have for some constant  $C_r$ depending only on $r$: 
\begin{equation}\label{grgr}  \forall n\in \mathbb N, \  \ \sharp \left\{\mathbf i^n\in \mathcal T_n,  \ (\sigma\circ \theta_{\mathbf i^n})_*\cap  B^G_{\sigma}(x,\epsilon, n)\neq \emptyset\right\}\leq C_r^n e^{\frac{w^n(\hat x)}{r-1}}.\end{equation}
\end{coro}

\begin{proof}[Sketch of proof]The Corollary follows from the Reparametrization Lemma together with the two following facts :
\begin{itemize}
\item for $y \in B^G_\sigma(x,\epsilon,n)$  we have  $k^n(x)\simeq k^n(y)$ up to $1$ on each coordinate, 
\item  for any $\mathbf i^{n-1}$ there is at most $C_re^{\frac{k(g^nx)-k'(G^n\hat x)}{r-1}}$  nodes $\mathbf i^n\in \overline{\mathcal T}_n$ with $\mathbf i_{n-1}^n=\mathbf i^{n-1}$ and $ \theta_{\mathbf i^n}([-1,1])\cap \sigma^{-1} B^G_{\sigma}(x,\epsilon, n+1)\neq \emptyset$.
\end{itemize}
 This last point  is a consequence of the last item of Lemma \ref{tech} applied to the bounded map $g\circ \psi\circ \tilde{\theta_i}$ introduced in the third step of the proof of the reparametrization lemma. 
\end{proof}

\subsection{The geometric set $E$}

We apply the Reparametrization Lemma to $g=f^p$ for some positive integer $p$. For $x\in \sigma_*$ we define the set $E_p(x)\subset p\mathbb N$ of integers    $mp$ such that there is $\mathbf i^m\in \overline{\mathcal T_m}$ with $k(\mathbf i^m)=k^m(x)$ and  $x\in \sigma \circ \theta_{\mathbf i^m}([-1/3,1/3])$. In particular, any  integer 
$n=mp\in E_p(x)$ is a $(\alpha_p, \epsilon_p)$-geometric time of $x$ for $f$ with $\alpha_p$ and $\epsilon_p$ depending only on $p$ by item (3) of RL. Observe that if $k<m$ we have with $x=\sigma\circ \theta_{\mathbf i^m}(t)$ and $\theta_{\mathbf i^m}(t)=\theta_{\mathbf i^m_k}(s)$:

\begin{align*}
\phi_{mp-kp}(F^{kp}\hat x)&=\frac{\left\|d\left(f^{mp}\circ \sigma \circ \theta_{\mathbf{i}^m}\right)(t)\right\|}{\left\|d\left(f^{kp}\circ \sigma \circ \theta_{\mathbf{i}^m}\right)(t)\right\|},\\
&\geq \frac{2}{3}\frac{\left\|d\left(f^{mp}\circ \sigma \circ \theta_{\mathbf{i}^m}\right)(0)\right\|}{\left\|d\left(f^{kp}\circ \sigma \circ \theta_{\mathbf{i}^m}\right)(t)\right\|}, \text{since $\sigma\circ \theta_{\mathbf i^m}$ is $m$-bounded},\\
&\geq \frac{2}{3}\frac{\left\|d\left(f^{mp}\circ \sigma \circ \theta_{\mathbf{i}^m}\right)(0)\right\|}{\left\|d\left(f^{kp}\circ \sigma \circ \theta_{\mathbf{i}^m_k}\right)(s)\right\|}100^{m-k}, \text{by item (2) of RL,}\\
&\geq \frac{2}{3\epsilon}\left\|d\left(f^{mp}\circ \sigma \circ \theta_{\mathbf{i}^m}\right)(0)\right\|100^{m-k}, \text{as $\sigma \circ \theta_{\mathbf i^n_k}$ is strongly $(k,\epsilon)$-bounded,}\\
&\geq \frac{1}{9}100^{m-k}\geq \left(\frac{1}{10}\right)^{m-k}, \text{ by item (3) of RL.}
\end{align*}
Therefore $E_p$ is $\tau_p$-large with $\tau_p=\frac{\log 10}{p}$-large.




  
  \begin{prop}\label{lebgeo} Let $f:M\circlearrowleft $ is a $C^r$ diffeomorphism and $b>\frac{R(f)}{r}$. For $p$ large enough there exists  $\beta_p>0$ such that 
  $$\limsup_n\frac{1}{n}\log \textrm{Leb}_{\sigma_*}\left(\left\{x\in A, \ d_{n}(E_{p}(x))< \beta_p \textrm{ and } \|d_xf^{n}(v_x)\|\geq e^{nb}\right\}\right)<0.$$
  \end{prop}
 \begin{proof}It is enough to consider $n=mp\in p\mathbb N$. We apply the Reparametrization Lemma to $g=f^p$ with $\epsilon>0$ being the  scale. Let $\mathcal T$ be the corresponding  tree and $(\theta_{\mathbf i^m})_{\mathbf i^m\in \mathcal T_m}$ its associated affine reparametrizations.   By a standard argument the number of sequences of positive  integers $(k_1,\cdots, k_m)$ with $k_i< M\in \mathbb N$  for all $i$ is less than $e^{mMH(M)}$  where  $H(t):=-\frac{1}{t}\log \frac{1}{t} -(1-\frac{1}{t})\log \left(1-\frac{1}{t}\right)$ for $t>0$.  
  Therefore we can fix the sequence $\mathbf k^m=k^{m}(x)$ up to a factor combinatorial term equal to $e^{2mpA_f H(pA_f)}$ with $A_f:=\log\|df\|_\infty+\log \|df^{-1}\|_\infty+1$. Assume there is $x=\sigma \circ \theta_{\mathbf i^m}(t)$ with $\|d_xf^{n}(v_x)\|\geq e^{nb}$. Then by the distorsion property of the bounded maps $f^n\circ\sigma \circ \theta_{\mathbf i^m}$ and $\sigma \circ \theta_{\mathbf i^m}$  we have 
  \begin{align*}
  |(\sigma\circ \theta_{\mathbf i^m})_*|&\leq 2\|d(\sigma\circ \theta_{\mathbf i^m})\|_\infty,\\
  &\leq  3\|d(\sigma\circ \theta_{\mathbf i^m})(t)\|, \text{ as $\sigma \circ \theta_{\mathbf i^m}$ is bounded,}\\
  &\leq 3 \frac{\|d(f^n\circ \sigma \circ \theta_{\mathbf i^m})(t)\|}{\|d_xf^{n}(v_x)\|} ,\\
  &\leq 3\epsilon e^{-nb}, \text{as $f^n\circ \sigma \circ \theta_{\mathbf i^m}$ is $\epsilon$-bounded.}
  \end{align*}

 Moreover when $x$ belongs to $ (\sigma\circ \theta_{\mathbf{i}^m})_*$ 
 for some $\mathbf{i}^m\in \mathcal T_m$ and    satisfies $d_{n}(E_{p}(x))<\beta_p$,  
 then we have $\sharp \left\{ 0<k<m, \ \mathbf{i}^m_k \in \overline{\mathcal T_k} \right\}\leq n\beta_p$. 
But, by the estimates on the valence of $\mathcal T$ given in the last item of RL,  the number of $m$-paths from the root labelled with  $\mathbf  k^m$ and with at most $n\beta_p $ red nodes are less 
 than $2^mC_r^{m}e^{\sum_i\frac{k_i-k'_i}{r-1}} \|df\|_\infty^{\beta_p p^2m}$ for some constant $C_r$ depending only on $r$. 
 Then if $x\in \mathcal H(\mathbf k^m)$ satisfies $\|d_xf^n(v_x)\|\geq e^{nb}$, we have $e^{\sum_i\frac{k_i-k'_i}{r-1}}\leq e^me^{m\frac{\log\|df^p\|_\infty-bp}{r-1}}$. But, as  $b$ is larger than $\frac{\log \|df^p\|_\infty}{pr}$ for large $p$, we get  for such values of $p$ :
 $\frac{\log\|df^p\|_\infty-bp}{r-1}\leq \frac{1-\frac{1}{r}}{r-1}\cdot \log \|df^p\|_\infty=\frac{\log \|df^p\|_\infty}{r}$. Therefore :
 \begin{align*} \limsup_n \frac{1}{n}\log\textrm{Leb}_{\sigma_*}\left(\left\{x\in A, \ d_{n}(E_{p}(x))< \beta_p \textrm{ and } \|d_xf^{n}(v_x)\|\geq e^{nb}\right\}\right)\\
 \leq \frac{\log C_r}{p}+\left(H(pA_f)+p\beta_p\right)A_f- b+\frac{\log \|df^p\|_\infty}{pr}.\end{align*}
 As $b$ is larger than $\frac{R(f)}{r}$ one can choose firstly $p\in \mathbb N^*$  large then  $\beta_p>0$ small  in such a way the right member is negative. 

 \end{proof}

From now we fix $p$ and the associated quantities satisfying the conclusion of  Proposition \ref{lebgeo} and we will simply write $E,\tau, \alpha, \epsilon, \beta$ for $E_p,\tau_p, \alpha_p, \epsilon_p, \beta_p$. The set $E(x)$ is called the \textbf{geometric set} of $x$.

\subsection{Cover of $F$-dynamical balls by bounded curves}
As a consequence of Corollary \ref{local}, we give now an estimate  of the number of strongly $(n,\epsilon')$-bounded curves  reparametrizing the intersection of a given strongly $\epsilon'$-bouded curve  with a $F$-dynamical ball of length $n$ and radius $\epsilon'$. This estimate will be used in the proof of the F\"olner Gibbs property (Proposition \ref{coeur}).\\  

For any $q\in \mathbb N^*$ we let $\omega_q:\mathbb PTM\rightarrow\mathbb R$ be the map defined  for all $(x,v)\in \mathbb P TM $  by 
$$\omega_q(x,v):=\frac{1}{q}\sum_{k=0}^{q-1}\log \|d_{f^kx} f^q\|-\log\|d_xf(v)\|.$$
Note that $\omega_1=w$. We also write $(\omega_q^n)_n$ the additive associated $F$-cocycle, i.e. 
$$\omega_q^n(x,v)=\sum_{0\leq k<n}\omega_q(F^k(x,v)).$$
\begin{lem}\label{annens}
For any $q\in \mathbb N^*$, there exists $\epsilon'_q>0$ and $B_q>0$ such that for any strongly $\epsilon'_q$-bounded curve $\sigma:[-1,1]\rightarrow M$, for  any  $x\in \sigma_*$ and  for any $n\in \mathbb N^*$ there exists a family $(\theta_i)_{i\in I_n}$ of affine maps of $[-1,1]$ such that :
\begin{itemize}
\item  $B_{\sigma}^F(x,\epsilon'_q, n)\subset \bigcup_{i\in I_n}(\sigma\circ \theta_i)_*$,
\item $\sigma\circ \theta_i$ is $(n,\epsilon'_q)$-bounded  (with respect to $f$) for any $i\in I_n$,
\item $\sharp I_n\leq B_q C_r^{\frac{n}{q}}e^{\frac{\omega_q^n(\hat x)}{r-1}},$ with $C_r$ a universal constant depending only on $r$.
\end{itemize}

\end{lem}

\begin{proof}
Fix $q$. Let $\epsilon'_q=\epsilon/2$ with $\epsilon$ as in Corollary \ref{local} for $g=f^q$. There is  a family $\Theta$ of affine maps of $[-1,1]$ such that for any strongly $\epsilon'_q$-bounded map $\gamma:[-1,1]\rightarrow M$, the map $\gamma \circ \theta $ is $(q,\epsilon'_q)$-bounded  and $\bigcup_{\theta\in \Theta}\theta_*=[-1,1]$. 

Fix now a strongly $\epsilon'_q$-bounded curve $\sigma:[-1,1]\rightarrow M$ and let $x\in \sigma_*$. We consider only the map $\theta\in \Theta$ such that $B_{\sigma}^F(x,\epsilon'_q,n)\cap (\sigma\circ\theta)_*\neq \emptyset$. For such a map $\theta$ we let $x_\theta \in B_{\sigma}^F(x,\epsilon'_q,n)\cap (\sigma\circ\theta)_*$.

 Take any  $0\leq k<q$. By applying Corollary \ref{local} to "$g=f^q$",  "$\sigma= f^k\circ \sigma \circ \theta$", "$x=f^{k}(x_\theta)$" and "$n=[\frac{n-k}{q}]$", we get a family $\Psi_{\theta, k}$ of affine maps of $[-1,1]$ with 
$$\bigcup_{\psi\in \Psi_{\theta,k}}(\sigma\circ \theta\circ \psi)_* \supset f^{-k}B_{ f^k\circ \sigma \circ \theta}^{F^q}\left(f^k(x_\theta),\epsilon,\left[\frac{n-k}{q}\right]\right)\supset  B_{ \sigma \circ \theta}^{F}(x,\epsilon'_q,n)$$ such that $f^{mq+k}\circ   \sigma \circ \theta\circ \psi$ is $\epsilon$-bounded for $\psi\in \Psi_{\theta,k}$ and integers $m$ with  $0\leq mq+k\leq n$.
  Then $\Theta_k=\{\theta\circ \psi\circ \theta', \ \psi\in  \Psi_{\theta, k} \text{ and } (\theta, \theta')\in \Theta^2\}$ satisfies the two first items of the conclusion. Moreover by letting $m_k=\left[\frac{n-k}{q}\right]$ we have: 
 \begin{align*}
 \sharp \Theta_k& \leq C_r^{m_k} \sharp \Theta^2 e^{w^{m_k}_{f^q}(F^k\hat x)}. 
 \end{align*}
 But for some constant $A_q$ depending only on $q$,  we have 
 $$\min_{0\leq k<q}e^{w^{m_k}_{f^q}(F^k\hat x)}\leq \left(\prod_{0\leq k<q} e^{w^{m_k}_{f^q}(F^k\hat x)}\right)^{1/q}\leq  A_q e^{ \omega_q^n(\hat x)}.$$
This concludes the proof of the lemma, as $\sharp \Theta$ depends only on $q$.

\end{proof}

\section{Existence of SRB measures}\label{srbb}

\subsection{Entropy formula}
By Ruelle's inequality \cite{Ruel}, for any $C^1$ system, the entropy of an invariant measure is less than or equal to the integral of the sum of its positive Lyapunov exponents. For $C^{1+}$ systems, the following entropy characterization of SRB measures was obtained by Ledrappier and Young :
\begin{theorem}\cite{led}\label{leddd}
An invariant   measure  of a $C^{1+}$ diffeomorphism on a compact manifold is a SRB measure if and only it has  a positive Lyapunov exponent almost everywhere and the  entropy is equal to the integral of the sum of its positive Lyapunov exponents. 
\end{theorem}

As the entropy is harmonic (i.e. preserves the ergodic decomposition), the ergodic components of a SRB measures are also SRB measures. 


\subsection{Lyapunov exponents}
We consider in this susection  a $C^1$ diffeomorphism $f:M\circlearrowleft$. Let $\|\|$ be a Riemaninan structure on $M$. The (forward upper) Lyapunov exponent of $(x,v)$ for  $x\in M$ and $v\in T_xM$ is defined as follows (see \cite{pes} for an introduction to Lyapunov exponents):
\[\chi(x,v):=\limsup_{n\rightarrow +\infty}\frac{1}{n}\log \|d_xf^n(v)\|.\]
The function $\chi(x,\cdot)$ admits only finitely many values $\chi_1(x)>...>\chi_{p(x)}(x)$ on 
$T_xM\setminus \{0\}$ and  generates a flag $ 0\subsetneq V_{p(x)}(x) \subsetneq \cdots
\subsetneq V_{1}=T_xM$ with  $V_i(x)=\{ v\in T_xM, \ \chi(x,v)\leq \chi_i(x)\}$. In particular, $\chi(x,v)=\chi_i(x)$ for $v\in V_i(x)\setminus V_{i+1}(x)$.  The function $p$  as well the functions $\chi_i$ and  the vector spaces $V_i(x)$, $i=1,...,p(x)$ are invariant and depend Borel measurably on $x$. One  can  show the maximal Lyapunov exponent $\chi$ introduced in the introduction coincides with $\chi_1$ (see Appendix A). 

A point $x$ is said \textbf{regular} when there exists a decomposition  $$T_xM=\bigoplus_{i=1}^{p(x)} H_i(x)$$ such that $$\forall v\in H_i(x)\setminus \{0\}, \ \lim_{n\rightarrow \pm\ \infty }\frac{1}{|n|}\log \|d_xf^n(v)\|=\chi_i(x)$$
with uniform convergence in $\{v\in H_i(x), \ \|v\|=1\}$ and $$\lim_{n\rightarrow \pm\ \infty }\frac{1}{|n|}\log \Jac \left(d_xf^n\right)=\sum_i \dim(H_i(x))\chi_i(x).$$
In particular we have $V_i(x)=\bigoplus_{j=1}^{i} H_j(x)$ for all $i$. The set $\mathcal R$ of regular points  is an invariant  measurable set of  full measure for any invariant measure  \cite{Os}.  The invariant subbundles $H_i$ are called the Oseledec's bundles.  We also let $\mathcal R^*:=\{x\in \mathcal R, \ \forall i \ \chi_i(x)\neq 0 \}$. For $x\in \mathcal R^*$ we put $E_u(x)= \bigoplus_{i, \ \chi_i(x)>0} H_i(x)$ and $E_s(x)= \bigoplus_{i, \ \chi_i(x)<0} H_i(x)$.\\ 

In the following we only consider surface diffeomorphisms. Therefore we always have $p(x)\leq 2$ and when $p(x)$ is equal to $1$, we let $\chi_2(x)=\chi_1(x)$.  
When $\nu$ is $f$-invariant we let $\chi_i(\nu)=\int \chi_i\, d\nu$. 


\subsection{Building SRB measures}\label{building}

 Assume $f$ is a $C^r$, $r>1$, surface diffeomorphism and $\limsup_n\frac{1}{n}\log \|d_xf^n\|>b>\frac{R(f)}{r}$ on a set of positive Lebesgue measure as in the Main Theorem. Take $\epsilon>0$ as in Proposition \ref{lebgeo} (it depends only on $b-\frac{R(f)}{r}>0$). 
By using Fubini's theorem  as in \cite{bur} there is a $C^r$ smooth embedded curve  $\sigma:I\rightarrow M$, which can be assumed to be $\epsilon$-bounded, 
and  a subset $A$ of $\sigma_*$ with $\Leb_{\sigma_*}(A)>0$, such that we have $\limsup_n\frac{1}{n}\log \|d_xf^n(v_x)\|>b$ for all $x\in A$. Here $\Leb_{\sigma_*}$ denotes the  Lebesgue measure on $\sigma_*$ induced by its inherited Riemannian structure as a submanifold of $M$. This a finite measure with $\Leb_{\sigma_*}(M)=|\sigma_*|$.  We can also assume that the countable set of periodic sources has an empty intersection with  $\sigma_*$.

  Let $\mathfrak m$ be the measurable sequence given by the set $\mathfrak m(x):=\{n, \ \|d_xf^n(v_x)\|\geq e^{nb}\}\in \mathfrak N $ for $x\in A$. It follows from Proposition \ref{lebgeo} that for $x$ in a subset of $ A$ of positive $\Leb_{\sigma_*}$ measure we have $d_n\left(E(x)\right)\geq \beta>0$  for $n\in \mathfrak m(x)$ large enough, i.e. we have  by denoting again this subset by $A$: 
$$\forall x\in A, \ \overline{d}(E(x))\geq  \overline{d}^{\mathfrak m}(E(x))\geq \beta.$$

For any $q\in \mathbb N^*$ we let 
$$\psi^q=\phi-\frac{\omega_q}{r-1}.$$

\begin{prop}\label{coeur}
There exists an infinite sequence of positive real numbers $(\delta_q)_{q}$  with $\delta_q\xrightarrow{q\rightarrow \infty} 0$ such that  the property $(H)$ holds with respect  to the additive cocycle  on $\mathbb PTM$ associated to the observable $\psi^q+\delta_q$.
\end{prop}

 We prove now  the existence of  a SRB measure  assuming  Proposition \ref{coeur}, whose proof is given in the next section.   This is  a first step in the proof of the Main Theorem. We will  apply the results of the first  sections  to the projective action $F:\mathbb PTM\circlearrowleft$ induced by $f$, where we consider:
\begin{itemize}
\item the additive derivative cocycle $\Phi=(\phi_k)_k$ given by $\phi_k(x,v)=\log \|d_xf^k(v)\|$,
\item   the measure $\lambda=\lambda_\sigma$ on $\mathbb PTM$ given by $ \mathfrak s^*\Leb_{\sigma_*}$ with $\mathfrak s:x\in \sigma_*\mapsto (x,v_x)$, 
\item  the  geometric set   $E$,  which is $\tau$-large with respect to $\Phi$, 
\item  the additive cocycles $\Psi^q$ associated to $\psi^q+\delta_q$.
\end{itemize}

 The topological extension   $\pi:(\mathbb P TM,F)\rightarrow (M,f)$ is principal\footnote{i.e. $h_f(\pi\mu)=h_F(\mu)$ for all $F$-invariant measure $\mu$.}    by a straightforward application of Ledappier-Walters variational principle \cite{led} and Lemma 3.3 in \cite{shub}. In fact this holds in any dimension and more generally for any finite dimensional vector bundle morphism instead of $df:TM\circlearrowleft$.

Let $\mathcal F=(F_n)_{n\in \mathfrak n}$ and $(A_n)_{n\in \mathfrak n}$ be the sequences associated to $E$ given by Lemma \ref{measurable}. Rigorously $E$ should be defined on the projective tangent bundle, but as 
$\pi$ is one-to-one on  $\mathbb PT\sigma_*$ there is no confusion. In the same way we see the sets $A_n$, $n\in \mathbb N$, as subsets of $A\subset\sigma_*$. 
Any weak-$*$ limit $\mu$ of $\mu_n^{F_n}$ is invariant under $F$ and thus supported by Oseledec's bundles. Let $\nu=\pi\mu$. By Lemma \ref{large}, $\mu$ is supported by the unstable bundle $E_u$ and $\phi_*(\hat x)\geq \tau>0$ for $\mu$ a.e. $\hat x\in \mathbb PTM$.  
Note also that $\phi_*(\hat x)\in \{\chi_1(\pi \hat x), \chi_2(\pi \hat x)\}$ for $\mu$-almost every $\hat x$. We claim that $\phi_*(\hat x)=\chi_1(\pi \hat x)$. If not $\nu$ would have an ergodic component with two positive exponents. It is well known such a measure is necessarily a periodic measure associated to a periodic source $S$. But there is an open neighborhood $U$ of the orbit of $S$ with $f^{-1}U\subset U$ and $\sigma_*\cap U=\emptyset$. In particular we have  $\pi\mu_n^{F_n}(U)=0$ for all $n$ because  $\pi\mu_n^{F_n}\left(\bigcup_{N\in \mathbb N}f^N\sigma_*\right)=1$ and $f^N\sigma_*\cap U=f^N(\sigma_*\cap f^{-N}U)\subset f^N(\sigma_*\cap U)=\emptyset$. By taking the limit in $n$ we would obtain $\nu(S)=0$. Therefore $\phi_*(\hat x)=\chi_1(\pi \hat x)>\tau$ for $\mu$-almost every $x$ and   $\chi_1(x)>\tau >0\geq \chi_2(x)$ for $\nu$-almost every $x$.

 Then by Proposition \ref{por} and Proposition \ref{coeur}  we obtain:
\begin{align*}\label{pa}h(\nu)=h(\mu)&\geq \int \psi^q\, d\mu +\delta_q,\nonumber\\
&\geq \int \phi\, d\mu- \frac{1}{r-1}\int \omega_q\, d\mu+ \delta_q,\\
& \geq \chi_1(\nu) -\frac{1}{r-1}\left( \frac{1}{q}\sum_{k=0}^{q-1}\int \log\|d_{f^kx}f^q\|\, d\nu(x)- \chi_1(\nu)\right)+\delta_q, \\
&\geq \chi_1(\nu)-\frac{1}{r-1}\left(\frac{1}{q}\int \log\|d_{x}f^q\|\, d\nu(x)-\chi_1(\nu)\right)+\delta_q.
\end{align*}
By a standard application of the subadditive ergodic theorem, we have 
$$\frac{1}{q}\int \log\|d_{x}f^q\|\, d\nu(x)\xrightarrow{q\rightarrow +\infty}\int \chi_1(x) \, d\nu(x)=\chi_1(\nu).$$
Therefore $h(\nu)\geq \chi_1(\nu)$, since $\delta_q\xrightarrow{q\rightarrow \infty}0$. Then Ruelle's inequality implies $h(\mu)= \chi_1(\nu)$. According to Ledrappier-Young characterization (Theorem \ref{leddd}), the measure 
 $\nu$ is a SRB measure of  $f$. Note also that any ergodic component $\xi$ of $\nu$ is also a SRB measure, therefore $h(\xi)=\chi_1(\xi)>\tau$. But by Ruelle inequality applied to $f^{-1}$, we get also $h(\xi)\leq -\chi_2(\xi)$. In particular we have $\chi_1(x)>\tau>0>-\tau>\chi_2(x)$ for $\nu$-almost every $x$. 

\begin{rem}
We have assumed that $r\geq 2$ is an integer. The proof can be adapted without difficulties  to the non integer case $r>1$. 
\end{rem}

\subsection{Proof of the F\"olner Gibbs property (H)}
In this subsection we prove Proposition \ref{coeur}.  We will show that for any $\delta>0$ 
there is $q$ arbitrarily large and $\epsilon'_q>0$  
such that we have for any partition $P$ of $\mathbb PTM$ with  diameter less than $\epsilon'_q$:
\begin{equation}\label{fred}\exists n_* \ \forall x\in A_n\subset \sigma_* \textrm{ with } n_*<n\in \mathfrak{n}, \  \ \  \frac{1}{\lambda_\sigma\left(P^{F_n}(\hat x)\cap \pi^{-1}A_n\right)}\geq e^{\delta \sharp F_n}e^{\psi_q^{F_n}(\hat x)}, \end{equation}
where we denote  $\psi_q^{F_n}(\hat x):=\sum_{k\in F_n}\psi^q(F^k\hat x)$ to simplify the notations.

For $G\subset \N$ we let $A^G$ be the set of points $x\in A$  with $G\subset E(x)$. When $G=\{ k\}$ or $\{k,l\}$ with $k,l\in \mathbb N$, we just let $A^G= A^k$ or $A^{k,l}$. We recall that  $\partial F_n\subset E(x)$ for all $x\in A_n$, in others terms $A_n\subset A^{\partial F_n}$.   We will show (\ref{ire}) for $A^{\partial F_n}$ in place  of $A_n$.

Fix the error term $\delta>0$. Let $q$ be so large that $C_r^{1/q}<e^{\delta/3}$ and $\epsilon'_q$ as in Lemma \ref{annens}.
 Without loss of generality we may  assume $\epsilon'_q<\frac{\alpha\epsilon}{4}$. Recall $\epsilon, \alpha>0$ corresponds to the fixed  scales in the definition of the geometric set $E$. 
  We can also ensure that\begin{equation}\label{ire}\forall  \hat x, \hat y \in \mathbb PTM \text{ with } \hat{\mathrm{d}}(\hat x, \hat y)<\epsilon'_q, \ \  |\phi(\hat x)-\phi(\hat y)|<\delta/3. \end{equation}  
 
  In the next triplets lemmas we consider a strongly  $\epsilon$-bounded  curve $\sigma$.

\begin{lem}\label{first}
For any subset $N$ of $M$, any $k\in \N$ and any ball $B_k$ of radius less than $\epsilon'_q$, there exists a finite family  $(y_j)_{j\in J}$ of 
$\sigma_*\cap A^k\cap f^{-k}B_k\cap N$ such that :
\begin{itemize}
\item $B_k\cap f^{k}(\sigma_*\cap A^k\cap N)\subset \bigcup_{j\in J}D_k(y_j)$, 
\item $B_k\cap D_k(y_j)$, $j\in J$, are pairwise disjoint.
\end{itemize}
\end{lem}

\begin{minipage}{.85\textwidth}%
\begin{proof}Let $y\in \sigma_*\cap A^k\cap N$ with $f^ky\in B=B_k$. Let $2B$ be the ball of same center as $B$ with twice radius. The curve $D_k(y)$ lies in  a  cone with opening angle $\pi/6$ by (\ref{oscill}).  Moreover its length is  larger than $\alpha\epsilon>4\epsilon'_q$. By elementary euclidean geometric arguments,   the set    $D_k(y)\cap 2B$ is a curve crossing  $2B$, i.e. its two endpoints lies in the boundary of $2B$. Two such  subcurves of $f^k\circ \sigma$ if not disjoint are necessarily equal.
\end{proof}
\end{minipage}%
\hfill
\begin{minipage}{.35\textwidth}%
\includegraphics[scale=0.2]{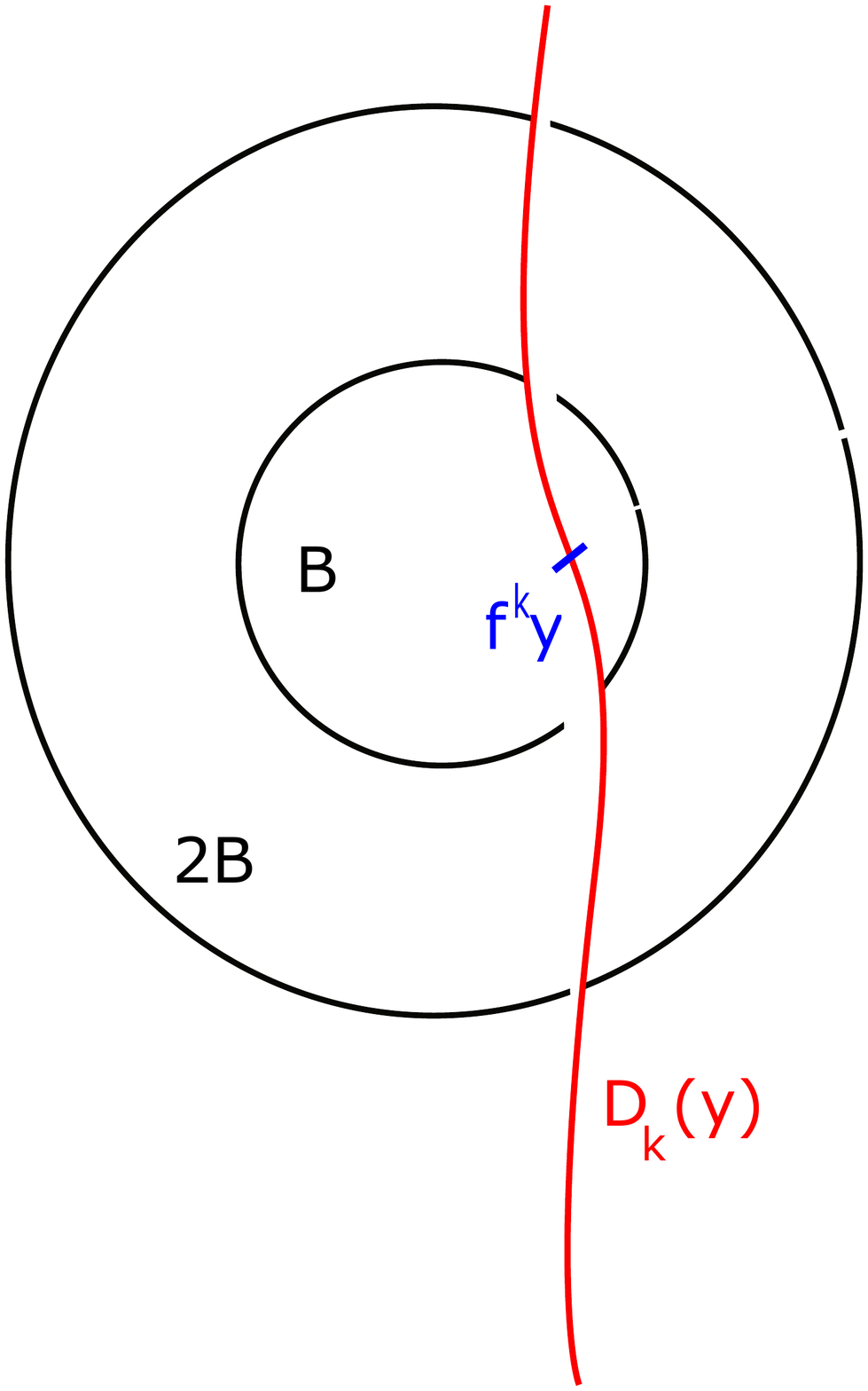}
\end{minipage}%

  As the distorsion is bounded on $D_k(y_i)$ by (\ref{distor}) and the semi-length of $D_k(y_i)$ is larger than $\alpha \epsilon$ (because  $y_i$ belongs to $A^k$), we have 
 \begin{align}\label{mar}
 \sum_{i\in I}\frac{4}{9}e^{-\phi_k (\hat y_i)}|D_k(y_i)|&\leq \sum_{i\in I}\left|f^{-k}D_k(y_i)\right|,\nonumber \\
 &\leq |\sigma_*|\leq 2\epsilon,\nonumber \\
  \sum_{i\in I}2\alpha\epsilon e^{-\phi_k (\hat y_i)} &\leq \frac{9}{2}\epsilon,\nonumber \\
 \sum_{i\in I}e^{-\phi_k (\hat y_i)} &\leq \frac{9}{4\alpha}.
 \end{align}

\begin{lem}\label{second}
For any subset $N$ of $M$ and any dynamical ball $B_{\G 0,k\R}:=B_\sigma^F(x,\epsilon'_q,k+1)$,  there exists a finite family  $(z_i)_{i\in I}$ of 
$\sigma_*\cap A^{k}\cap B_{\G 0,k\R}\cap N$ such that 
\begin{itemize}
\item $f^{k}\left(\sigma_*\cap A^{k}\cap B_{\G 0,k\R}\cap N\right)\subset \bigcup_{i\in I}D_k(z_i)$, 
\item $B(f^kx, \epsilon'_q)\cap D_k(z_i)$, $i\in I$, are pairwise disjoint, 
\item $\sharp I \leq B_qe^{\delta k/3}e^{\frac{\omega_q^k(\hat x)}{r-1}}$ for some constant $B_q$ depending only on $q$.
\end{itemize}
\end{lem}
\begin{proof}
As in the previous lemma we consider the subcurves $D_k(z)$ for $z \in \sigma_*\cap A^{k}\cap B_{\G 0,k\R}\cap N$. By Lemma \ref{annens} we can reparametrize $B_{\G 0,k\R}$ by a family of  strongly  $(k, \epsilon'_q)$-bounded curves with cardinality less than $B_qC_r^{\frac{k}{q}}e^{\omega_q^k(\hat x)}$. Each of this curve is contained in some $D_k(z)$. But as already mentioned,  the sets $B(f^kx,\epsilon'_q)\cap D_k(z)$, $z\in  \sigma_*\cap A^{k}\cap B_{\G 0,k\R}$,  are either disjoint or equal. 
\end{proof}

\begin{lem}\label{third}
For any dynamical ball $B_{\G k,l\R}:=f^{-k}B_{\sigma}^F(f^kx,\epsilon'_q,l-k+1)$,  there exists a finite family  $(y_i)_{i\in I}$ of 
$\sigma_*\cap A^{k,l}\cap B_{\G k,l\R}$ and a partition $I=\coprod_{j\in J}I_j$ of $I$ with $j\in I_j$ for all $j\in J\subset I$  such that  
\begin{itemize}
\item $ f^{l}(\sigma_*\cap A^{k,l}\cap B_{\G k,l\R})\subset \bigcup_{i\in I}D_l(y_i)$, 
\item $B(f^lx,\epsilon'_q)\cap D_l(y_i)$, $i\in I$, are pairwise disjoint, 
\item $\forall j\in J \ \forall i,i'\in I_j, \ D_k(y_i)\cap B(f^kx,\epsilon'_q)=D_k(y_{i'})\cap B(f^kx,\epsilon'_q)$,
\item $\forall j\in J, \ \sharp I_j \leq B_qe^{\delta (l-k)/3}e^{\frac{\omega_q^{l-k}(F^k\hat x)}{r-1}}$  for some constant $B_q$ depending only on $q$.
\end{itemize}
\end{lem}

\begin{proof}We first apply Lemma \ref{first} to $\sigma$ and $N=A^{k,l}\cap B_{\G k,l\R}$ to get the collection of 
strongly $\epsilon$-bounded curve $\left(D_k(y_j)\right)_{j\in J}$.   Then we apply Lemma \ref{second} to each $ D_k(y_j)$ for $j\in J$ and $N=f^{k}(B_{\G k,l\R}\cap A_k\cap \sigma_*)$ to get a family $(z_i)_{i\in I_j}$ of  $ D_k(y_j)\cap A^{l-k}\cap f^{k}( B_{\G k,l\R}\cap A_k\cap \sigma_*)$  satisfying:\begin{figure}[!ht]
\includegraphics[scale=0.4]{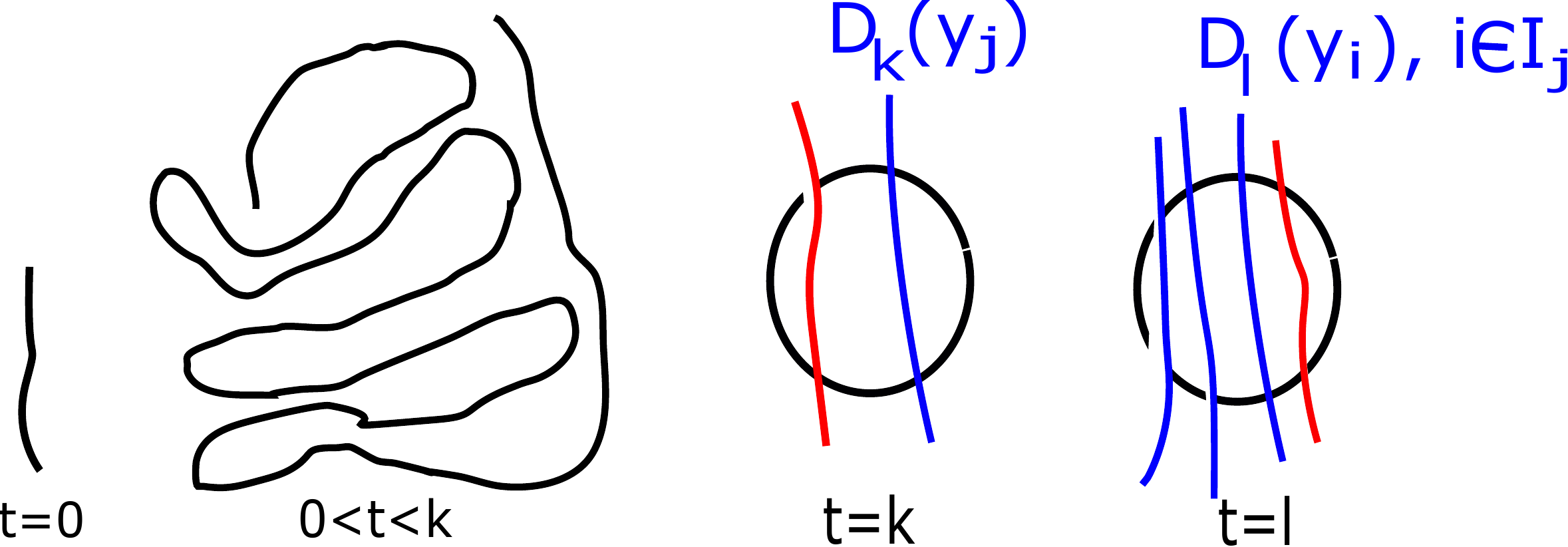}
\centering
\caption{\label{bl}\textbf{} For $0\leq t<k$ the image of $f^t\circ \sigma$ in black may be large  and the disks $D_t(y_i)$ are scattered  through the surface. For $t=k$, 
the sets $D_k(y_j)$ for $j\in J$ are covering $(f^t\circ \sigma)_*\cap B_k$. 
For $t=l$, we drew in blue the sets $D_l(y_i)\subset f^{l-k}D_k(y_j)$ for $i\in I_j$.}
\end{figure}  

\begin{itemize}
\item $f^{l-k}\left(D_k(y_j)\cap A^{l-k}\cap f^{k}(B_{\G k,l\R}\cap A_k\cap \sigma_*)\right)\subset \bigcup_{j\in J}D_k(z_i)$, 
\item $B_l\cap D_{l-k}(z_i)$, $i\in I_j$, are pairwise disjoint, 
\item $\sharp I_j \leq B_q e^{\delta (l-k)/3}e^{\frac{\omega_q^{l-k}(F^k\hat x)}{r-1}}$.
\end{itemize}

For all $j\in J$ we can take $j\in I_j$ and $z_j=f^k(y_j)$. 
We conclude the proof by letting $y_i=f^{-k}z_i\in \sigma_*\cap A^{k,l}\cap B_{\G k,l\R}$ for all $i\in I:=\coprod_{j\in J}I_j$.

\end{proof}

We prove now $(H)$.  Recall that $\lambda=\lambda_\sigma$ is the push-forward on $\mathbb PTM$ of the Lebesgue measure on $\sigma_*$. As $\sharp \partial F_n=o(n)$ it  is enough to show  there is a constant $C$ such that for any strongly $\epsilon$-bounded curve $\sigma$ we have 
\begin{equation}\label{fdf}\lambda_\sigma\left(P^{F_n}(\hat x)\cap \pi^{-1}A^{\partial F_n}\right)\leq C^{\sharp \partial F_n}e^{2\delta \sharp F_n/3}e^{-\psi_q^{F_n}(\hat x)}.
\end{equation}
To prove (\ref{fdf}) we argue by induction on the number of connected components of $F_n$.   Let $\G k,l\R$,  $0\leq k\leq l$, be the first  connected component of $F_n$ and write $G_{n-l}= \N^*\cap (F_n-l)$. 
Then with the notations of Lemma \ref{third} we get 
\begin{align*}
\lambda_{\sigma}\left(P^{F_n}(\hat x)\cap \pi^{-1}A^{\partial F_n}\right)&\leq  \lambda_\sigma\left( \coprod_{i\in I}F^{-l}\left(\pi^{-1}A^{\partial G_{n-l}}\cap P^{G_{n-l}}(F^l\hat x)\cap D_l(y_i)\right)\right), \\
&\leq \lambda_{\sigma}\left( \coprod_{j\in J} F^{-k} \left(\coprod_{i\in I_j}F^{-(l-k)}\left(\pi^{-1}A^{\partial G_{n-l}}\cap P^{G_{n-l}}(F^l\hat x)\cap D_l(y_i)\right)\right)\right)
\end{align*}

For $j\in J$ we let  $\sigma_j^k$ be the strongly $\epsilon$-bounded curve $\sigma$ given by $D_k(y_j)$. By the bounded distorsion property (\ref{distor}) we get 
\begin{align*}
\lambda_{\sigma}\left(P^{F_n}(\hat x)\cap \pi^{-1} A^{\partial F_n}\right)&\leq 3\sum_{j\in  J}e^{-\phi_k(\hat y_j)}\lambda_{\sigma_j^k}\left(\coprod_{i\in I_j}F^{-(l-k)}\left(\pi^{-1} A^{\partial G_{n-l}}\cap P^{G_{n-l}}(F^l\hat x)\cap D_l(y_i)\right)\right).
\end{align*}
By using again the bounded distorsion property (now between the times $k$ and $l$) we get 
with $\sigma_i^l$ being the curve associated to $D_l(y_i)$ :
\begin{align*}
\lambda_{\sigma}\left(P^{F_n}(\hat x)\cap \pi^{-1}A^{\partial F_n}\right)\leq & 9\sum_{j\in  J}e^{-\phi_k(\hat y_j)}\sum_{i\in I_j} e^{-\phi_{l-k}(F^k\hat y_i)} \lambda_{\sigma_i^l}\left(\pi^{-1}A^{\partial G_{n-l}}\cap P^{G_{n-l}}(F^l\hat x)\right).
\end{align*}

We may assume that  any $\hat y_i$, $i\in I$, lies in $P^{F_n}(\hat x)$. In particular we have  $|\phi_{l-k}(F^k \hat y_i)-\phi_{l-k}(F^k\hat x)|<(l-k)\delta/3$ by (\ref{ire}). Then 
\begin{align*}
\lambda_{\sigma}\left(P^{F_n}(\hat x)\cap \pi^{-1}A^{\partial F_n}\right)\leq &9\left(\sum_{j\in  J}e^{-\phi_k(\hat y_j)}\right) e^{\delta(l-k)/3} e^{-\phi_{l-k}(F^k\hat x)}\sup_j \sharp I_j \\
&\times \sup_{i\in I}\lambda_{\sigma_i^l}\left(\pi^{-1}A^{\partial G_{n-l}}\cap P^{G_{n-l}}(F^l\hat x)\right).
\end{align*}

By (\ref{mar}) and the last item of Lemma \ref{third} we obtain 
\begin{align*}
\lambda_{\sigma}\left(P^{F_n}(\hat x)\cap \pi^{-1}\partial A^{F_n}\right)&\leq \frac{50B_qe^{2\delta (l-k)/3}}{\alpha}e^{-\phi_{l-k}(F^k\hat x)+\frac{\omega_q^{l-k}(F^k\hat x)}{r-1}}\sup_{i\in I}\lambda_{\sigma_i^l}\left(\pi^{-1}A^{\partial G_{n-l}}\cap P^{G_{n-l}}(F^l\hat x)\right), \\
& \leq \frac{50B_qe^{2\delta (l-k)/3}}{\alpha}e^{-\psi_q^{\G k,l\R}(\hat x)}\sup_{i\in I}\lambda_{\sigma_i^l}\left(\pi^{-1}A^{\partial G_{n-l}}\cap P^{G_{n-l}}(F^l\hat x)\right).
\end{align*}

By induction hypothesis (\ref{fdf}) applied to $G_{n-l}$ for each $\sigma_i^l$,  we have for all $i\in I$:
$$\lambda_{\sigma_i^l}\left(\pi^{-1}A^{\partial G_{n-l}}\cap P^{G_{n-l}}(F^l\hat x)\right)\leq C^{\sharp \partial G_{n-l}} e^{2\delta\sharp \partial G_{n-l}/3}e^{-\psi_q^{G_{n-l}}(F^l\hat x)}.$$
Note that $\sharp \partial F_n=\sharp \partial G_{n-l}+2$. We conclude by taking $C=\sqrt{ \frac{50B_q}{\alpha}}$  that 
\begin{align*}
\lambda_{\sigma}\left(P^{F_n}(\hat x)\cap \pi^{-1}A^{\partial F_n}\right)&\leq \frac{50B_qe^{2\delta \sharp F_n /3}}{\alpha}C^{\sharp \partial G_{n-l}}e^{-\psi_q^{F_n}(\hat x)},\\
&\leq C^{\sharp \partial F_n} e^{2\delta  \sharp F_n/3}e^{-\psi_q^{F_n}(\hat x)}.
\end{align*}
This completes the proof of (\ref{fred}).

\section{End of the proof of the Main Theorem}
\subsection{The covering property of the basins} 
For $x\in M$  the stable/unstable manifold $W^{s/u}(x)$ at $x$ are defined as follows :
$$W^s(x):=\{y\in M, \ \limsup_{n\rightarrow +\infty}\frac{1}{n}\log d(f^n x,f^n y)<0\},$$
$$W^u(x):=\{y\in M, \ \limsup_{n\rightarrow +\infty}\frac{1}{n}\log d(f^{-n} x,f^{-n} y)<0\}.$$

For a subset $\Gamma$ of $M$ we let $W^s(\Gamma)=\bigcup_{x\in \Gamma}W^s(x)$.
According to Pesin's theory, there are a nondecreasing sequence of  compact, a priori non-invariant, sets $(K_n)_n$ (called the Pesin blocks) with $\mathcal R^*=\bigcup_n K_n$ and two families of embedded $C^\infty$ discs $(W^s_{loc}(x))_{x\in K}$ and $(W^u_{loc}(x))_{x\in K}$ (called the local stable and unstable manifolds) such that :
\begin{itemize}
\item  $W^{s/u}_{loc}(x)$ are tangent to $E_{s/u}$ at $x$,
\item the splitting $E_u(x)\oplus E_s(x)$ and the discs $W^{s/u}_{loc}(x)$ are continuous on $x\in K_n$ for each $n$.
\end{itemize}

For $\gamma>0$ and $x\in K$ we let $W^{s/u}_\gamma(x)$ be the connected component of $B(x,\gamma)\cap W^{s/u}_{loc}(x)$ containing $x$.

\begin{prop}\label{bass}
The set $\left\{\chi>\frac{R(f)}{r}\right\}$ is covered by the basins of ergodic SRB measures $\mu_i$, $i\in I$, up to a set of zero Lebesgue measure. 
\end{prop}
 
In fact we prove a  stronger statement by showing that $\left\{\chi>\frac{R(f)}{r}\right\}$  is contained Lebesgue a.e. in $W^s(\Gamma)$ where  $\Gamma$ is any $f$-invariant subset of $\bigcup_{i\in I}\mathcal B(\mu_i)_{i\in I}$ with $\mu_i(\Gamma)=1$ for all $i$. \\

So far we only have used the characterization of SRB measure in terms of entropy (Theorem \ref{leddd}). In the proof of Proposition \ref{bass} we will use the absolutely continuity property of SRB measures. Let $\mu$ be a Borel measure on $M$. We recall a measurable partition $\xi$ in the sense of Rokhlin \cite{rok} is said $\mu$-subordinate to $W^u$ when $\xi(x)\subset W^u(x)$ and $\xi(x)$ contains an open neighborhood of $x$ in the topology of $W^u(x)$  for 
$\mu$-almost every $x$. The measure $\mu$ is said to have \textbf{absolutely continuous conditional measures on unstable manifolds} if for every measurable partition $\xi$ $\mu$-subordinate to $W^u$, the conditional measures $\mu_x^\xi$  of $\mu$ with respect to $\xi$ satisfy   $\mu_x^\xi \ll  Leb_{W^u(x)}$ for $\mu$-almost every  $x$.

\begin{proof}
We argue by contradiction. Take $\Gamma$ as above. Assume there is a Borel set $B$ with positive Lebesgue measure contained in the complement of $W^s(\Gamma)$ such that we have  $\chi(x)>b>\frac{R(f)}{r}$  for all $x\in B$. Then we follow the approach of Section \ref{building}. We  consider  a $C^r$ smooth disc $\sigma$ with $\chi(x,v_x)>b$ for $x\in B'\subset B$, $\Leb_{\sigma_*}(B')>0$. 
One can then  define the geometric set $E$ on a subset $B''$ of $B$ with $\Leb_{\sigma_*}(B'')>0$.  We also let $\tau$, $\beta$, $\alpha$ and $\epsilon$ be the  parameters associated to $E$. Recall that :
\begin{itemize}
\item  $E$ is $\tau$-large with respect to the derivative cocycle $\Phi$,
\item $\overline{d}(E(x))\geq \beta >0$ for $x\in B''$, 
 \item $D_k(y)=f^k(H_k(y))$ has semi-length larger than $\alpha\epsilon$ when $k\in E(y)$, $y\in B''$.
 \end{itemize}

Let $B'''$ be the subset of $B''$ given by density points of $B''$ with respect to $\Leb_{\sigma_*}$. In particular, we have  
$$\forall x\in B''',  \ \ \frac{\textrm{Leb}_{\sigma_*}\left( H_k(x)\cap B''\right) }{\textrm{Leb}_{\sigma_*}(H_k(x))}\xrightarrow{k\rightarrow +\infty}1.$$

We choose  a subset $A$  of $B'''$  with $\Leb_{\sigma_*}(A)>0$ such that the above convergence is uniform in $x\in A$.  Then from this set $A$ and  the geometric set $E$ on $A$  we may  build $\mathfrak n$,  $(F_n)_{n\in \mathfrak n}$ and $(\mu_n^{F_n})_{n\in \mathfrak n}$  as in Sections \ref{drei}  and \ref{zwei}.  
   As proved in Section \ref{building} any limit measure $\mu$ of $\mu_n^{F_n}$ is supported on the unstable bundle and projects  to a SRB measure $\nu$  with $\chi_1(x)\geq\tau>0>-\tau\geq \chi_2(x)$ for $\nu$ a.e. $x$.  The measure  $\nu$ is a barycenter of   ergodic SRB measures with such  exponents.  Take $P=K_N$ a Pesin block with $\nu(P)\sim 1>1-\beta$. We let $\theta$ and $l$ be respectively  the minimal angle between $E_u$ and $E_s$  and   the minimal  length of the  local stable and unstable manifolds  on $P$.    

Let $\xi$ be a measurable partition subordinate to $W^u$ with diameter less then $\gamma>0$. We have $\nu(\Gamma\cap P)=\int \nu_x^\xi(\Gamma\cap P)\, d\nu(x)\sim 1$ and $\nu_x^\xi \ll  Leb_{W_\gamma^u(x)}$ for $\nu$ a.e. $x$. Therefore we get   for some $c>0$ 
$$\nu\left(x, \ \textrm{Leb}_{W_\gamma^u(x)}(\Gamma\cap P)>c\right)\sim 1 .$$

We let $F=\{x\in \Gamma\cap P, \ \Leb_{W_\gamma^u(x)}(\Gamma\cap P)>c\}$. Observe that we have again $\nu(F)\sim 1$.  For $x\in \sigma_*$ and $y\in P$ we use the following notations :
$$\hat x_\sigma=(x,v_x)\in \mathbb PT\sigma_* \ \ \ \ \ \hat y_u=(y,v_y^u)\in \mathbb PTM,$$ where $v_y^u$ is the element of $\mathbb PTM$ representing the line $E_u(y)$. Let $\hat F_u^\gamma$ be the open $\gamma/8$-neighborhood of $\hat F_u:=\{\hat y_u, \ y \in F \}$ in $\mathbb PTM$. Recall $E(x)$ denotes the set of geometric times of $x$.  We let for $n\in \mathfrak n$: $$\zeta_n:=\int \frac{1}{\sharp F_n }\sum_{k\in  E(x)\cap F_n}\delta_{F^k \hat x_\sigma } \, d\mu_n(\hat x_\sigma ).$$ Observe that $
\zeta_n(\mathbb PTM)\geq \inf_{x\in A_n} d_n(E(x)\cap F_n) $. By the last item in Lemma \ref{measurable}, we have $\liminf_{n\in \mathfrak n}\inf_{ x\in A_n} d_n(E(x)\cap F_n)\geq\beta$. Therefore there is a weak limit $\zeta=\lim_k\zeta_{p_k}$ with $\zeta \leq \mu$ and $\zeta(\mathbb PTM) \geq \beta$. From $\mu(\hat F_u^\gamma)\sim 1>1-\beta$ we get $0<\zeta(\hat F_u^\gamma)\leq \lim_k \zeta_{p_k}(\hat F_u^\gamma)$. Note also $\hat A_\sigma:=\{\hat y_\sigma, \ y\in A \} $ has full $\mu_n$-measure for all $n$. In particular,  for infinitely many $n\in \mathbb N$ there is $(x^n,v_{x^n})=\hat x^n_\sigma 
\in \hat A_\sigma $ with $F^n\hat x^n_\sigma\in \hat F_u^{\gamma}$ and $n\in E(x^n)$.  Let $\hat y^n_u=(y^n,v^u_{y^n})\in \hat F_u$ which is $\gamma/
8$-close to $F^n\hat x^n_\sigma$.  Then for $\gamma\ll \delta\ll \min(\theta, l)$ independent of $n$, the curve $D_n^\delta(x^n):=D_n(x^n)\cap B(f^nx^n,\delta)$ is transverse to $W^s(P\cap \Gamma\cap W^u_{\gamma}
(y^n))$ and  may be written as the graph of a $C^r$ smooth function $\psi:E\subset E_u(y^n)\rightarrow E_s(y^n)$ with $\|d\psi\|<L$ for a universal constant $L$.

\begin{figure}[!ht]
\includegraphics[scale=0.4]{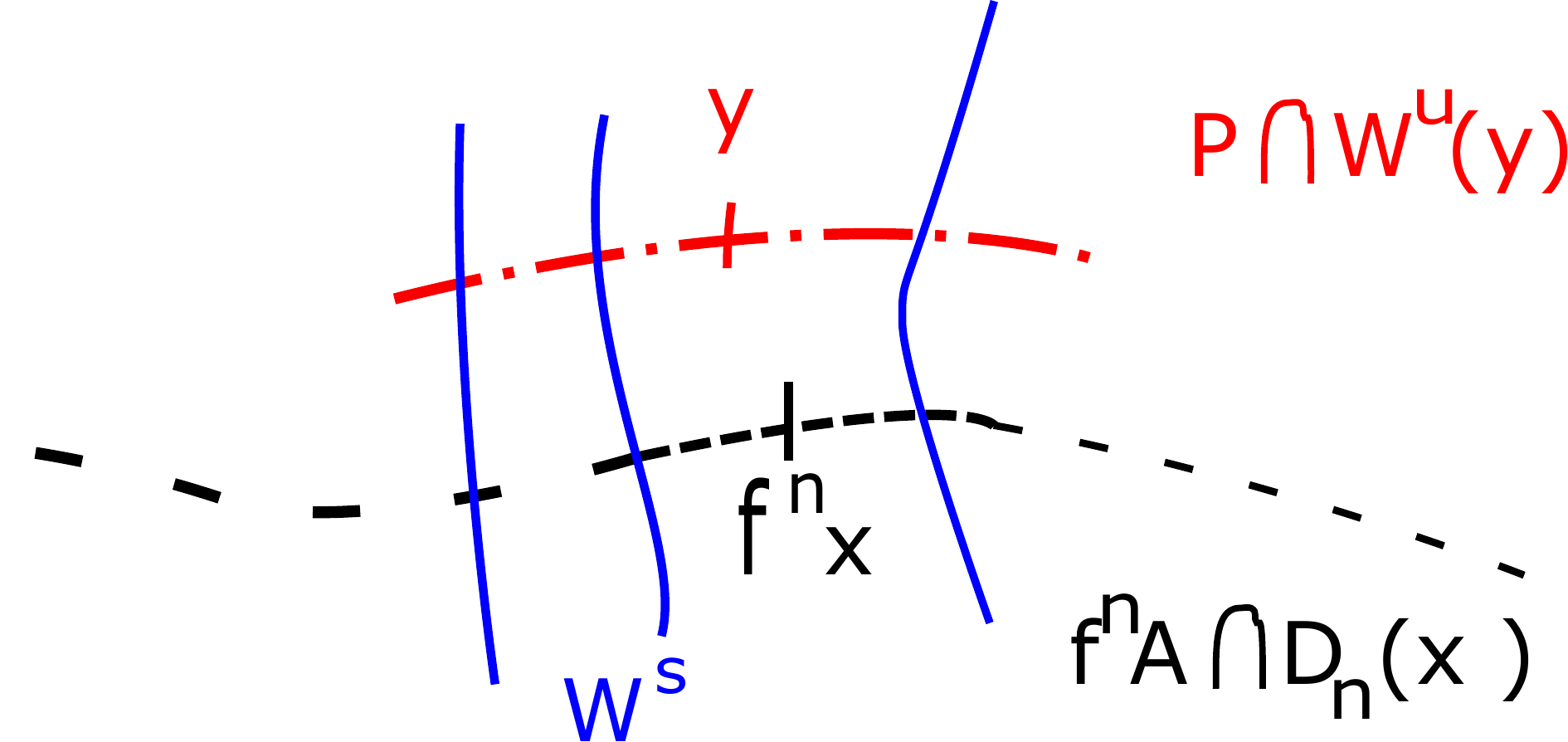}
\centering
\caption{\label{bl}\textbf{Holonomy of the local stable foliation between the transversals $D_n(x)$ and $W^u_\gamma(y)$.} }
\end{figure}

By Theorem 8.6.1 in \cite{pes} the  associated holonomy map  $h: W^{u}_\gamma(y)\rightarrow D_n^\delta(x^n)$ is absolutely continuous and its Jacobian is  bounded from below by a positive constant  depending only on the Pesin block $P=K_N$ (not on $x^n$ and $y^n$). In particular $\Leb_{D_n(x^n)}\left(W^s(\Gamma \cap P) \right)\geq c'$ for some constant $c'$ independent of $n$. 
The distorsion of $df^n$ on $H_n(x^n)$ being  bounded by $3$, 
we get (recall $f^nH_n(x^n)=D_n(x^n)$):
$$ (\alpha\epsilon)^{-1}\Leb\left(D_n(x^n)\setminus f^n B\right)\leq \frac{\Leb\left(D_n(x^n)\setminus f^n B\right)}{\Leb\left(D_n(x^n)\right)}\leq 9 \frac{\Leb\left(H_n(x^n)\setminus B \right) }{\Leb\left(H_n(x^n)\right)} \xrightarrow{n\rightarrow \infty}0.$$

 Therefore for $n$ large enough, there exists $x\in f^{n}B\cap W^s(\Gamma \cap P)$, in particular 
 $B\cap  f^{-n} W^s(\Gamma)=B\cap  W^s(\Gamma)\neq \emptyset$. This contradicts the definition of $B$.

\end{proof}

\subsection{The maximal exponent}

Let $\mathcal R^{+*}$ denote the invariant subset of Lyapunov regular points $x$ of $(M,f)$ with $\chi_1(x)>0>\chi_2(x)$. Such a  point admits  so called regular neighborhoods (or $\epsilon$-Pesin charts):

\begin{lem}\cite{ppes}\label{chart}
For a fixed $\epsilon>0$ there exists a measurable function $q=q_\epsilon:\mathcal R^{+*}\rightarrow (0,1]$ with $e^{-\epsilon}<q(fx)/q(x)<e^{\epsilon}$ and a collection of embeddings $\Psi_x:B(0,q(x))\subset T_xM=E_u(x)\oplus E_s(x)\sim \mathbb R^2\rightarrow M$ with $\Psi_x(0)=x$  such that $f_x=\Psi_{fx}^{-1}\circ f\circ \Psi_x$ satisfies the following properties :
\begin{itemize}
\item  \begin{equation*}d_0f_x =
\begin{pmatrix} 
a_\epsilon^1(x) &0\\
0& a_\epsilon^2(x),
\end{pmatrix}
\end{equation*}
with $e^{-\epsilon}e^{\chi_i(x)}<a_\epsilon^i(x) <e^{\epsilon}e^{\chi_i(x)}$  for $i=1,2$,
\item the $C^1$ distance between $f_x$ and $d_0f_x$ is less than $\epsilon$, 
\item there exists a constant $K$ and a measurable function $A=A_\epsilon:\mathcal R^{+*}\rightarrow \mathbb R$ such that for all $y,z\in B(0,q(x))$ 
$$Kd(\Psi_x(y),\Psi_x(z))\leq \|y-z\|\leq A(x)d(\Psi_x(y), \Psi_x(z)),$$
with $e^{-\epsilon}A(fx)/A(x)<e^{\epsilon}$.
\end{itemize}
\end{lem}

For any $i\in I$ we let 
$$E_i:=\{x, \ \chi(x)=\chi(\mu_i)\}.$$ 
The set $E_i$ has  full $\mu_i$-measure by the subadditive ergodic theorem. Put $\Gamma_i=\mathcal B(\mu_i)\cap E_i\cap \mathcal R^{+*}$ and $\Gamma=\bigcup_{i}\Gamma_i$. Clearly $\Gamma$ is $f$-invariant.  We finally check that 	$\chi(x)=\chi(\mu_i)$ for $x\in W^s(\Gamma_i)$. 

For uniformly hyperbolic systems, we have $$\Sigma \chi(x)=\lim_n\frac{1}{n}\log \Jac(d_xf^n_{E_u})=\lim_n \int \log \Jac(d_yf_{E_u})\, d\delta_x^n.$$ As the geometric potential $y\mapsto \log \Jac d_yf_{E_u}$ is continuous in this case, any point in the basin of a SRB measure $\mu$ satisfies $\Sigma \chi(x)=\int \Sigma \chi(y)\, d\mu(y)$.

 \begin{lem}\label{fini} If $y\in W^s(x)$ with $x\in \mathcal R^{+*}$, then $\chi(y)=\chi(x)$.
\end{lem}
\begin{proof}  Fix $x\in \mathcal R^{+*}$ and $\delta>0$. 
 We  apply Lemma \ref{chart} with $\epsilon\ll \chi_1(x) $.  For $\alpha>0$ we let $\mathcal C_\alpha$ be the cone
$\mathcal C_\alpha=\{(u,v)\in \mathbb R^2, \ \alpha\|u\|\geq \|v\|\}$. We may choose $\alpha>0$ and $\epsilon>0$ so small that for all $k\in \mathbb N$ we have   $d_zf_{f^kx}(\mathcal C_\alpha)\subset \mathcal C_\alpha$ and $\|d_zf_{f^kx}(v)\|\geq e^{\chi_1(x)-\delta}$ for all $v\in\mathcal C_\alpha$ and all $z\in B(0, q_\epsilon(f^kx))$. 

Let $y\in W^s(x)$. There is $C>0$ and $\lambda$ such that $d(f^nx,f^ny)<C\lambda^n$ holds for all $n\in \mathbb N$. We can choose $\epsilon\ll \lambda$. In particular there is  $N>0$ such that $f^ny$ belongs to $\Psi_{f^nx} B(0, q(f^nx))$ for $n\geq N$ since we have  $A(f^nx)<e^{\epsilon n}A(x)$ and $q(f^nx)>e^{\epsilon n}q(x)$. Let $z \in B(0,q(f^Nx))$ with $\Psi_{f^Nx}(z)=y$. Then for all $v\in \mathcal C_\alpha$ and for all $n\geq N$ we have  
 $\|d_z\left(\Psi_{f^{n-N}x}^{-1}\circ f^{n-N}\circ \Psi_{f^Nx} \right)(v)\|\geq e^{(n-N)(\chi_1(x)-\delta)}.$
 As the conorm  of $d_{f^{n-N}y}\psi_{f^nx}$ is bounded from above by  $A(f^n x)^{-1}$ for all $n$ we get 
\begin{align*}
\chi(y)&= \limsup_n\frac{1}{n}\log \|d_{y} f^{n-N}\|,\\
&=  \limsup_n\frac{1}{n}\log \|d_z \left(f^{n-N}\circ \Psi_{f^Nx}\right)\|,\\
& \geq \limsup_{n\rightarrow +\infty}\frac{1}{n}\log \left(  A(f^nx)^{-1}\left\|d_z\left(\Psi_{f^nx}^{-1}\circ f^n\circ \Psi_{f^Nx} \right) \right\|\right),\\
&\geq \chi_1(x)-\delta-\epsilon. 
\end{align*}
On the other hand  we have 
 \begin{align*}
 \left\|d_z\left(\Psi_{f^nx}^{-1}\circ f^n\circ \Psi_{f^Nx} \right) \right\|&\leq \prod_{k=N}^{n-1}\sup_{t\in B(0, q(f^kx))}\|d_{t}f_{f^kx}\|,\\
 &\leq \left(e^{\chi_1(x)+\epsilon}+\epsilon\right)^{n-N},\\
 &\leq e^{(n-N)(\chi_1(x)+2\epsilon)}. 
 \end{align*} 
Then it follows from  $\|d_{f^{n-N}y}\psi_{f^nx}\|\leq K$:
 
\begin{align*}
 \chi(y)&\leq  \limsup_{n\rightarrow +\infty}\frac{1}{n}\log \left(\left\|d_z\left(\Psi_{f^nx}^{-1}\circ f^n\circ \Psi_{f^Nx} \right) \right\|\right), \\
 &\leq \chi_1(x)+2\epsilon.
 \end{align*}
 
As it holds for arbitrarily small  $\epsilon$ and $\delta$ we get $\chi(y)=\chi_1(x)=\chi(x)$.

\end{proof}

We  conclude with $\Lambda=\{\chi(\mu_i), \ i\in I\}$ that for Lebesgue a.e. point $x$, we have $\chi(x)\in \left]-\infty, \frac{R(f)}{r}\right]\cup \Lambda$ and that $\{\chi=\lambda\}\stackrel{o}{\subset}\bigcup_{i\in I, \, \chi(\mu_i)=\lambda}\mathcal B(\mu_i)$ for all $\lambda\in \Lambda$. The proof of the Main Theorem is now complete. It follows also from Lemma \ref{fini}, that the converse statement of Corollary \ref{coco} holds : if $(f,M)$ admits a SRB measure then $\Leb(\chi>0)>0$.\\

In the  proof of Lemma \ref{fini} we can choose the cone $\mathcal C_\alpha$ to be contracting, so that any vector in a small cone at $y$ will converge to the unstable direction $E_u(x)$. In other terms if we endow the smooth manifold $\mathbb PTM$ with a smooth Riemanian structure, then the lift $\mu$ to the unstable bundle of an ergodic SRB measure $\nu$ is a physical measure of $(\mathbb PTM, F)$. Conversely if $\mu$ is a physical measure of $(\mathbb PTM, F)$ supported on the unstable bundle above $\mathcal R^{+*}$, we can reproduce the scheme of the proof of the Main Theorem to show $\mu$ projects to a SRB measure $\nu$. Indeed we may consider a $C^\infty$ smooth curve  $\sigma$ such that $\hat x=(x,v_x)$ lies in the basin of $\mu$ for $\hat x$ in a positive Lebesgue measure set $A$ of $\hat \sigma_*$. Then by following the above  construction  of SRB measures, we obtain that $\mu=\lim_{n}\frac{1}{n}\sum_{k=0}^{n-1}F_*^k\Leb_A$ project to  a SRB measure (or one can directly use the approach of \cite{bure}). This converse statement is very similar to a result of Tsujii (Theorem A in \cite{tsu}) which states in dimension two  that for $C^{1+}$ surface diffeomorphism  an ergodic hyperbolic measure $\nu$, such that the set of \textbf{regular} points $x\in \mathcal B(\nu)$ with $\chi(x)=\int \chi\, d\nu$ has positive Lebesgue measure, is a SRB measure.  Indeed if $\mu$ is a physical measure of $(\mathbb PTM, F)$ supported on the unstable bundle above $\mathcal R^{+*}$, its projection $\nu$ satisfies $\chi(\nu)=\chi(x)$ for any $x\in \pi(\mathcal B(\nu))$. In the present paper we are working  with the stronger $C^\infty$ assumption, but, in return, points in the basin are not supposed to be regular contrarily to Tsujii's theorem. 

From the above discussion, we may restate  the Main Theorem as follows:
\begin{theorem}
Let $f:M\circlearrowleft$ be a $C^\infty$ surface diffeomorphism and $F:\mathbb PTM$ be the induced map on the projective tangent bundle. Then the basins of the physical measures of $(\mathbb PTM,F)$ are covering Lebesgue almost everywhere the set $\{(x,v)\in \mathbb PTM, \ \chi(x,v)>0\}$.  
\end{theorem}

\section{Nonpositive  exponent in  contracting sets}\label{gui}

In this last section  we show Theorem \ref{yoyo}. For a dynamical system $(M,f)$ a subset $U$ of $M$ is said \textbf{almost contracting} when for all $\epsilon>0$ the set $E_\epsilon=\{k\in \mathbb N, \ \diam(f^k U)>\epsilon\}$ satisfies $\overline{d}(E_\epsilon)=0$. 
In \cite{gui} the authors build subsets with historic behaviour and positive Lebesgue measure which are almost  contracting but not contracting.
We will show Theorem \ref{yoyo} for almost contracting sets. 
\\

We  borrow the next  lemma  from \cite{bure} (Lemma 4 therein), which may be stated with the notations of Section \ref{srbb} as follows :
\begin{lem}\label{dav} Let  $f:M\circlearrowleft$ be a $C^\infty$ diffeomorphism admitting and let    $U$ be a subset of $M$ with $\Leb\left(\left\{\chi>a\right\}\cap U\right)>0$ for some $a>0$.
 Then for all $\gamma>0$  there is a $C^\infty$ smooth embedded curve  $\sigma_*$ 
and $I\subset \mathbb N$ with $\sharp I=\infty$ such that
\[\forall n\in I, \ \left| \{x\in U\cap \sigma_*, \, \| d_xf^n(v_x)\|>e^{na} \} \right|>e^{-n\gamma}.\]
 \end{lem}

We are now in a position to prove Theorem \ref{yoyo} for almost contracting sets. 
\begin{proof}[Proof of Theorem \ref{yoyo}]
 We argue by contradiction by assuming $\Leb\left(\left\{\chi>a\right\}\cap U\right)>0$ for some $a>0$ with $U$ being a almost  contracting set. By Yomdin's Theorem  on one-dimensional local volume growth for $C^\infty$ dynamical systems \cite{yom} there is  $\epsilon>0$ so small that 
\begin{equation}\label{volu} v^*(f,\epsilon)=\sup_{\sigma}\limsup_{n\to\infty}\frac{1}{n}\sup_{x\in M}\log \left|f^n(B(x,\epsilon,n)\cap \sigma_*\right|<a/2,
\end{equation} where the supremum holds over all $C^{\infty}$ smooth embedded curves $\sigma:[0,1]\rightarrow M$. 
As $U$ is almost contracting, there are  subsets $(C_n)_{n\in \mathbb N}$ of $M$ with $\lim_n\frac{\log \sharp C_n}{n}=0$ satisfying for all $n$ 
\begin{equation}\label{rat}
U\subset \bigcup_{x\in C_n}B(x,\epsilon,n).
\end{equation}

Fix an error term  $\gamma\in ]0,\frac{a}{2}[$. Then  by Lemma \ref{dav} there is a $C^\infty$-smooth curve $\sigma_*\subset U$ and an infinite subset $I$  of $\mathbb N$ such that  for all $n\in I$
\begin{eqnarray*}
\sum_{x\in C_n}\left| f^n(B(x,\epsilon,n)\cap \sigma_*)\right|&\geq &\left| f^n(U\cap \sigma_*)\right|,\\ 
&\geq & e^{na}\left| \{x\in U\cap\sigma_*, \, \| d_xf^n(v_x)\|>e^{na} \}\right|,\\
&\geq & e^{n(a-\gamma)} \text{ by    (\ref{volu}),}\\
\sharp C_n \sup_{x\in M}\log \left|f^n(B(x,\epsilon,n)\cap \sigma_*)\right| & \geq & e^{n(a-\gamma)} \text{ by    (\ref{rat}).}\\
\end{eqnarray*}
Therefore we get the contradiction $v^*(f,\epsilon)>a-\gamma>a/2$.
\end{proof}

\appendix
\section{}
Let $\mathcal A=(A_n)_{n\in \mathbb N}$ be a sequence in $M_d(\mathbb R^d)$.  
For any $n\in \mathbb{N}$ we let $A^n=A_{n-1}\cdots A_1A_0$. We define the  Lyapunov exponent $\chi(\mathcal A)$ of $\mathcal{A}$ with respect to  $v\in \mathbb{R}^d\setminus \{0\}$ as 
$$\chi(\mathcal A,v):=\limsup_n\frac{1}{n}\log \|A^n(v)\|,$$

\begin{lem}
$$\sup_{v\in \mathbb{R}^d\setminus \{0\}}\chi(\mathcal A,v)=\limsup_n\frac{1}{n}\log \vertiii{A^n}.$$
\end{lem}
\begin{proof}The inequality $\leq $ is obvious. Let us show the other inequality. Let $v_n\in \mathbb{R}^d$ 
with $\|v_n\|=1$ and  $\|A^n(v_n)\|= \vertiii{A^n}$. Then take $v=\lim_{k}v_{n_k}$ with $\lim_{k}\frac{1}{n_k}\log\vertiii{A^{n_k}}=\limsup_n\frac{1}{n}\log \vertiii{A^n}$. We get 
\begin{align*}
\|A^{n_k}(v)\|&\geq \|A^{n_k}(v_k)\|- \|A^{n_k}(v-v_k)\|,\\
&\geq \vertiii{A^{n_k}}(1-\|v-v_k\|),\\
\limsup_k\frac{1}{n_k}\log\|A^{n_k}(v)\|&\geq \limsup_n\frac{1}{n}\log \vertiii{A^n}.
\end{align*}
\end{proof}



\end{large}

\end{document}